\documentclass[12pt]{article}

\usepackage{amssymb,amsthm,amsmath,amsfonts,mathrsfs}%General
\usepackage{bm}
\usepackage{cancel}
\usepackage{enumerate}
\usepackage{url}
\usepackage{chessboard}
\usepackage{pst-all}

\newcommand{\GL}{{G^\mathcal{L}}}%Games
\newcommand{\GR}{{G^\mathcal{R}}}
\newcommand{\N}{\mathcal{N}}
\renewcommand{\P}{\mathcal{P}}
\renewcommand{\L}{\mathcal{L}}
\newcommand{\R}{\mathcal{R}}

\newtheorem{theorem}{Theorem}%Environments

\newtheorem{lemma}[theorem]{Lemma}
\newtheorem{definition}[theorem]{Definition}

\newtheorem{observation}[theorem]{Observation}

\emergencystretch=\maxdimen
\newcommand{\custombrace}[2]{%   %Custom braces and slash
  \left\{
  \begin{array}{@{}c@{}}
    \rule{0pt}{#1}\\
    #2
  \end{array}
  \right.
}

\newcommand{\custombracee}[2]{%
  \left\}
  \begin{array}{@{}c@{}}
    \rule{0pt}{#1}\\
    #2
  \end{array}
  \right.
}

\newcommand{\customslash}[2]{%
  \makebox[0pt][l]{\rule[-#1]{#2}{#1}}|
}

\title{A complete solution for the\\ {\sc partisan chocolate game}}

\begin{document}

%\begin{frontmatter}

\author{Tomoaki Abuku\thanks{buku3416@gmail.com, National Institute of Informatics} \and Hikaru Manabe\thanks{urakihebanam@gmail.com, Keimei Gakuin Junior and High School} \thanks{Hikaru Manabe's work is funded by Kadokawa Dwango Educational Institute, Academic Research Club.} \and Richard J. Nowakowski \thanks{r.nowakowski@dal.ca, Dalhousie University} \and 
Carlos P. Santos \thanks{cmf.santos@fct.unl.pt, Center for Mathematics and Applications (NovaMath), FCT NOVA} \thanks{Carlos Santos' work is funded by national funds through the FCT - Funda\c{c}\~{a}o para a Ci\^{e}ncia e a Tecnologia, I.P., under the scope of the projects UIDB/00297/2020 and UIDP/00297/2020 (Center for Mathematics and \linebreak Applications).} \and
Koki Suetsugu\thanks{suetsugu.koki@gmail.com, National Institute of Informatics} \\
}

%\author{Tomoaki Abuku \\National Institute of Informatics\thanks{buku3416@gmail.com }}
%\author{Hikaru Manabe \\ Keimei Gakuin Junior and High School\thanks{urakihebanam@gmail.com}\thanks{Hikaru Manabe's work is funded by Kadokawa Dwango Educational Institute, Academic Research Club.}

%\author{Richard J. Nowakowski \\ Dalhousie University\thanks{r.nowakowski@dal.ca}}
%\author{Carlos P. Santos \\ Center for Mathematics and Applications (NovaMath), FCT NOVA \thanks{cmf.santos@fct.unl.pt}\thanks{Carlos Santos' work is funded by national funds through the FCT - Funda\c{c}\~{a}o para a Ci\^{e}ncia e a Tecnologia, I.P., under the scope of the projects UIDB/00297/2020 and UIDP/00297/2020 (Center for Mathematics and \linebreak Applications).}

%\fntext[fn1]{Carlos Santos' work is funded by national funds through the FCT - Funda\c{c}\~{a}o para a Ci\^{e}ncia e a Tecnologia, I.P., under the scope of the projects UIDB/00297/2020 and UIDP/00297/2020 (Center for Mathematics and \linebreak Applications).}
%\author{Koki Suetsugu \\ National Institute of Informatics\thanks{suetsugu.koki@gmail.com}}

\begin{abstract}
The class of Poset Take-Away games includes many interesting and difficult games. Playing on an $n$-dimensional positive quadrant (the origin being the bottom of the poset) gives rise to {\sc nim, wythoff’s nim} and {\sc chomp}. These are impartial games. We introduce a partisan game motivated by {\sc chomp} and the recent chocolate-bar version. Our game is played on a chocolate bar with alternately flavored pieces (or a checkerboard). We solve this game by showing it is equivalent to {\sc blue-red hackenbush} strings. This equivalence proves that the values of game are numbers and it gives an algorithm for optimal play when there is more than one chocolate bar. The checkerboard interpretation leads to many natural questions.
\end{abstract}

Combinatorial Game Theory, chocolate-bar games, Jacobsthal Sequence.

\vspace{0.4cm}
MSC 2020 91A46, 91A05, 91A80

%\end{frontmatter}

\maketitle

\section{Introduction}
\label{sec:intro}
We assume that the reader is acquainted with the basic concepts of short two-person perfect information combinatorial games (Combinatorial Game Theory, CGT) as presented in any of \cite{ANW007,BCG82,Con76,Sie013}. We only consider normal play, i.e. the player who cannot move loses. Indeed, the background required to follow the proofs presented in this document is minimal and we have opted to include a small subsection in the final part of this introduction. This way, the paper becomes fully self-contained. Readers fluent in CGT may wish to skip that subsection.

Chocolate-bar formulations are common in CGT. A notable example is the ruleset {\sc chomp}, played on a rectangular grid consisting of smaller square cells, which can be thought of as the blocks of a chocolate bar. The\linebreak chocolate-bar formulation of {\sc chomp} is attributed to David Gale \cite{Gale74}, but an equivalent ruleset, expressed in terms of choosing divisors of a fixed integer, was published earlier by Frederik Schuh \cite{Sch52}. This classical ruleset still abounds with interesting open problems.

Recently, new chocolate-bar formulations have been introduced in \cite{Rob89, MN015, MNN020, MN021}. In the impartial, two-dimensional version, there is a poisoned square in the bottom-left corner. A legal move consists of breaking the bar along a horizontal or vertical line, and eating the portion that does not contain the poisoned square\footnote{Naturally, players do not want to die.}. Here, we introduce a natural two-dimensional \emph{partisan} version, the {\sc partisan chocolate game}, and provide a comprehensive solution for it, expressed in Theorem \ref{thm:main}. This constitutes the main result of this document.

The rules of partisan chocolate game are the following: (1) a component consists of a rectangular grid (a chocolate bar) composed of smaller square cells which, except for the bottom-left, are colored in a checkerboard fashion of blue (blueberry) and red (strawberry) squares. The bottom-left is poisoned and colored black, and the squares adjacent to it orthogonally are blue. (2) Left\footnote{Traditional female and positive player in CGT.} can cut the chocolate along a vertical line, provided that the top square of the column immediately to the right of that line is blue. Left can also cut the chocolate along a horizontal line, as long as the rightmost square in the row just above that line is blue. Right\footnote{Traditional male and negative player in CGT.} can make the same type of moves, but provided that the top square of the column immediately to the right of the line or the rightmost square in the row just above the line is red. After making a cut, a player eats the portion of the chocolate that does not contain the poisoned square.

We may assume that the two squares adjacent to the poisoned square are blue, since the case where they are red is the negative. It should be noted that the game can be played with multiple chocolate bars simultaneously. In this case, traditional disjunctive sum rules apply, where a player, on their turn, chooses one of the chocolate bars and makes
their move there, leaving all other components unchanged. An $n \times m$ chocolate bar will be denoted as $(n-1, m-1)$, where $n$ is the number of columns, and $m$ is the number of rows. The leftmost column and the bottommost row are indexed as zero. Figure \ref{fig:fig1} shows the $(2,3)$-position, along with the options for Left and Right (after they have made their moves and eaten the delicious portion).

\begin{figure}[htbp]
\begin{center}
\scalebox{0.47}{
\begin{tikzpicture}
\clip(0.5,0.5) rectangle (29.5,5.5);

\draw [line width=2.pt] (1.,1.)-- (1.,5.);
\draw [line width=2.pt] (1.,1.)-- (4.,1.);
\draw [line width=2.pt] (4.,1.)-- (4.,5.);
\draw [line width=2.pt] (1.,5.)-- (4.,5.);
\draw [line width=2.pt] (2.,1.)-- (2.,5.);
\draw [line width=2.pt] (3.,1.)-- (3.,5.);
\draw [line width=2.pt] (1.,4.)-- (4.,4.);
\draw [line width=2.pt] (1.,3.)-- (4.,3.);
\draw [line width=2.pt] (1.,2.)-- (4.,2.);
\fill[line width=0.pt,fill=black,fill opacity=1.0] (2.,1.) -- (2.,2.) -- (1.,2.) -- (1.,1.) -- cycle;

\draw [line width=2.pt] (8.,1.)-- (8.,5.);
\draw [line width=2.pt] (9.,1.)-- (9.,5.);
\draw [line width=2.pt] (10.,1.)-- (10.,5.);
\draw [line width=2.pt] (10.,5.)-- (8.,5.);
\draw [line width=2.pt] (10.,4.)-- (8.,4.);
\draw [line width=2.pt] (10.,3.)-- (8.,3.);
\draw [line width=2.pt] (8.,2.)-- (10.,2.);
\draw [line width=2.pt] (8.,1.)-- (10.,1.);
\draw (8.1,2.9) node[anchor=north west] {\scalebox{1.5}{\textcolor[rgb]{0.00,0.07,1.00}{\textbf{B}}}};
\draw (8.1,2.9+2) node[anchor=north west] {\scalebox{1.5}{\textcolor[rgb]{0.00,0.07,1.00}{\textbf{B}}}};
\draw (8.1+1,2.9+1) node[anchor=north west] {\scalebox{1.5}{\textcolor[rgb]{0.00,0.07,1.00}{\textbf{B}}}};
\draw (8.1+1,2.9-1) node[anchor=north west] {\scalebox{1.5}{\textcolor[rgb]{0.00,0.07,1.00}{\textbf{B}}}};
\draw (8.1+1,2.9) node[anchor=north west] {\scalebox{1.5}{\textcolor[rgb]{1.00,0.00,0.00}{\textbf{R}}}};
\draw (8.1,2.9+1) node[anchor=north west] {\scalebox{1.5}{\textcolor[rgb]{1.00,0.00,0.00}{\textbf{R}}}};
\draw (8.1+1,2.9+2) node[anchor=north west] {\scalebox{1.5}{\textcolor[rgb]{1.00,0.00,0.00}{\textbf{R}}}};
\fill[line width=0.pt,fill=black,fill opacity=1.0] (9.,1.) -- (9.,2.) -- (8.,2.) -- (8.,1.) -- cycle;

\draw [line width=2.pt] (12.,3.)-- (15.,3.);
\draw [line width=2.pt] (12.,4.)-- (12.,1.);
\draw [line width=2.pt] (12.,1.)-- (15.,1.);
\draw [line width=2.pt] (14.,1.)-- (14.,4.);
\draw [line width=2.pt] (13.,4.)-- (13.,1.);
\draw [line width=2.pt] (15.,4.)-- (15.,1.);
\draw [line width=2.pt] (15.,2.)-- (12.,2.);
\draw [line width=2.pt] (12.,4.)-- (15.,4.);
\draw (8.1+5,2.9-1) node[anchor=north west] {\scalebox{1.5}{\textcolor[rgb]{0.00,0.07,1.00}{\textbf{B}}}};
\draw (8.1+4,2.9) node[anchor=north west] {\scalebox{1.5}{\textcolor[rgb]{0.00,0.07,1.00}{\textbf{B}}}};
\draw (8.1+6,2.9) node[anchor=north west] {\scalebox{1.5}{\textcolor[rgb]{0.00,0.07,1.00}{\textbf{B}}}};
\draw (8.1+5,2.9+1) node[anchor=north west] {\scalebox{1.5}{\textcolor[rgb]{0.00,0.07,1.00}{\textbf{B}}}};
\draw (8.1+5,2.9) node[anchor=north west] {\scalebox{1.5}{\textcolor[rgb]{1.00,0.00,0.00}{\textbf{R}}}};
\draw (8.1+6,2.9-1) node[anchor=north west] {\scalebox{1.5}{\textcolor[rgb]{1.00,0.00,0.00}{\textbf{R}}}};
\draw (8.1+6,2.9+1) node[anchor=north west] {\scalebox{1.5}{\textcolor[rgb]{1.00,0.00,0.00}{\textbf{R}}}};
\draw (8.1+4,2.9+1) node[anchor=north west] {\scalebox{1.5}{\textcolor[rgb]{1.00,0.00,0.00}{\textbf{R}}}};
\fill[line width=0.pt,fill=black,fill opacity=1.0] (13.,1.) -- (13.,2.) -- (12.,2.) -- (12.,1.) -- cycle;

\draw [line width=2.pt] (25.-7.5,2.)-- (28.-7.5,2.);
\draw [line width=2.pt] (25.-7.5,1.)-- (28.-7.5,1.);
\draw [line width=2.pt] (25.-7.5,1.)-- (25.-7.5,2.);
\draw [line width=2.pt] (26.-7.5,1.)-- (26.-7.5,2.);
\draw [line width=2.pt] (27.-7.5,1.)-- (27.-7.5,2.);
\draw [line width=2.pt] (28.-7.5,1.)-- (28.-7.5,2.);
\draw (8.1+18-7.5,2.9-1) node[anchor=north west] {\scalebox{1.5}{\textcolor[rgb]{0.00,0.07,1.00}{\textbf{B}}}};
\draw (8.1+19-7.5,2.9-1) node[anchor=north west] {\scalebox{1.5}{\textcolor[rgb]{1.00,0.00,0.00}{\textbf{R}}}};
\fill[line width=0.pt,fill=black,fill opacity=1.0] (26.-7.5,1.) -- (26.-7.5,2.) -- (25.-7.5,2.) -- (25.-7.5,1.) -- cycle;

\draw [line width=2.pt] (18.+5.5,1.)-- (18.+5.5,5.);
\draw [line width=2.pt] (18.+5.5,4.)-- (17.+5.5,4.);
\draw [line width=2.pt] (18.+5.5,3.)-- (17.+5.5,3.);
\draw [line width=2.pt] (17.+5.5,1.)-- (17.+5.5,5.);
\draw [line width=2.pt] (17.+5.5,5.)-- (18.+5.5,5.);
\draw [line width=2.pt] (18.+5.5,1.)-- (17.+5.5,1.);
\draw (8.1+9+5.5,2.9) node[anchor=north west] {\scalebox{1.5}{\textcolor[rgb]{0.00,0.07,1.00}{\textbf{B}}}};
\draw (8.1+9+5.5,2.9+2) node[anchor=north west] {\scalebox{1.5}{\textcolor[rgb]{0.00,0.07,1.00}{\textbf{B}}}};
\draw (8.1+9+5.5,2.9+1) node[anchor=north west] {\scalebox{1.5}{\textcolor[rgb]{1.00,0.00,0.00}{\textbf{R}}}};
\fill[line width=0.pt,fill=black,fill opacity=1.0] (18.+5.5,1.) -- (18.+5.5,2.) -- (17.+5.5,2.) -- (17.+5.5,1.) -- cycle;

\draw [line width=2.pt] (20.+5.5,3.)-- (23.+5.5,3.);
\draw [line width=2.pt] (20.+5.5,2.)-- (23.+5.5,2.);
\draw [line width=2.pt] (23.+5.5,1.)-- (20.+5.5,1.);
\draw [line width=2.pt] (20.+5.5,1.)-- (20.+5.5,4.-1);
\draw [line width=2.pt] (21.+5.5,4.-1)-- (21.+5.5,1.);
\draw [line width=2.pt] (22.+5.5,1.)-- (22.+5.5,4.-1);
\draw [line width=2.pt] (23.+5.5,4.-1)-- (23.+5.5,1.);
\draw (8.1+12+5.5,2.9) node[anchor=north west] {\scalebox{1.5}{\textcolor[rgb]{0.00,0.07,1.00}{\textbf{B}}}};
\draw (8.1+14+5.5,2.9) node[anchor=north west] {\scalebox{1.5}{\textcolor[rgb]{0.00,0.07,1.00}{\textbf{B}}}};
\draw (8.1+13+5.5,2.9-1) node[anchor=north west] {\scalebox{1.5}{\textcolor[rgb]{0.00,0.07,1.00}{\textbf{B}}}};
\draw (8.1+13+5.5,2.9) node[anchor=north west] {\scalebox{1.5}{\textcolor[rgb]{1.00,0.00,0.00}{\textbf{R}}}};
\draw (8.1+14+5.5,2.9-1) node[anchor=north west] {\scalebox{1.5}{\textcolor[rgb]{1.00,0.00,0.00}{\textbf{R}}}};
\fill[line width=0.pt,fill=black,fill opacity=1.0] (21.+5.5,1.) -- (21.+5.5,2.) -- (20.+5.5,2.) -- (20.+5.5,1.) -- cycle;

\draw (5,3.3) node[anchor=north west] {\scalebox{3}{$=$}};
\draw (7,5.6) node[anchor=north west] {$\custombrace{4.8cm}{}$};
\draw (21.4,8) node[anchor=north west] {$\customslash{10cm}{1pt}$};
\draw (28.5+0.5,5.6) node[anchor=north west] {$\custombracee{4.8cm}{}$};
\draw (10.7,1.2) node[anchor=north west] {\scalebox{3}{$,$}};
\draw (10.7+13.5,1.2) node[anchor=north west] {\scalebox{3}{$,$}};
\draw (15.9,1.2) node[anchor=north west] {\scalebox{3}{$,$}};

\draw (8.1-7,2.9) node[anchor=north west] {\scalebox{1.5}{\textcolor[rgb]{0.00,0.07,1.00}{\textbf{B}}}};
\draw (8.1-5,2.9) node[anchor=north west] {\scalebox{1.5}{\textcolor[rgb]{0.00,0.07,1.00}{\textbf{B}}}};
\draw (8.1-6,2.9-1) node[anchor=north west] {\scalebox{1.5}{\textcolor[rgb]{0.00,0.07,1.00}{\textbf{B}}}};
\draw (8.1-6,2.9+1) node[anchor=north west] {\scalebox{1.5}{\textcolor[rgb]{0.00,0.07,1.00}{\textbf{B}}}};
\draw (8.1-7,2.9+2) node[anchor=north west] {\scalebox{1.5}{\textcolor[rgb]{0.00,0.07,1.00}{\textbf{B}}}};
\draw (8.1-5,2.9+2) node[anchor=north west] {\scalebox{1.5}{\textcolor[rgb]{0.00,0.07,1.00}{\textbf{B}}}};

\draw (8.1-6,2.9) node[anchor=north west] {\scalebox{1.5}{\textcolor[rgb]{1.00,0.00,0.00}{\textbf{R}}}};
\draw (8.1-5,2.9-1) node[anchor=north west] {\scalebox{1.5}{\textcolor[rgb]{1.00,0.00,0.00}{\textbf{R}}}};
\draw (8.1-7,2.9+1) node[anchor=north west] {\scalebox{1.5}{\textcolor[rgb]{1.00,0.00,0.00}{\textbf{R}}}};
\draw (8.1-5,2.9+1) node[anchor=north west] {\scalebox{1.5}{\textcolor[rgb]{1.00,0.00,0.00}{\textbf{R}}}};
\draw (8.1-6,2.9+2) node[anchor=north west] {\scalebox{1.5}{\textcolor[rgb]{1.00,0.00,0.00}{\textbf{R}}}};
\end{tikzpicture}}
\caption{The $(2,3)$-position and its options.}
\label{fig:fig1}
\end{center}
\end{figure}
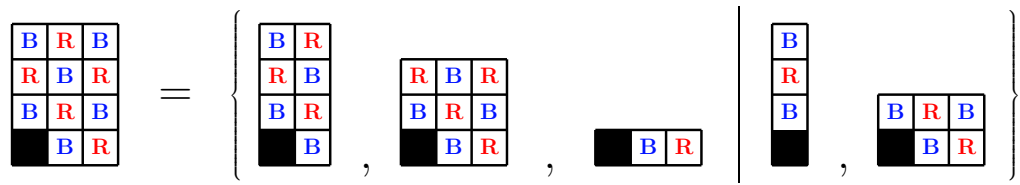

\subsection{Background}
\label{subsec:background}

\vspace{0.4cm}
 Using standard notation, where Left and Right are the players, a position is written in the \emph{form} $G=\{\GL \! \mid \! \GR \}$, where $\GL=\{G^{L_1},G^{L_2},\ldots\}$ is the set of Left \emph{options} from $G$ and $G^{L}$ is a particular Left option (the same for the Right options $\GR$).

 The possible {\em outcomes} of a position are $\mathscr{L}$, $\mathscr{N}$, $\mathscr{P}$, and $\mathscr{R}$. From Left's perspective, the first is the best (she wins, regardless of whether playing first or second) and the fourth is the worst (she loses, regardless of whether playing first or second). On the other hand, regarding $\mathscr {N}$ and $\mathscr{P}$, the victory depends on playing first or second, so these outcomes are not comparable. These considerations explain the partial order of outcomes:
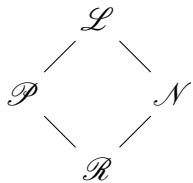
\begin{center}
\begin{tikzpicture}[node distance=1cm]
  \node (L) at (0,2) {$\mathscr{L}$};
  \node (P) at (-1,1) {$\mathscr{P}$};
  \node (N) at (1,1) {$\mathscr{N}$};
  \node (R) at (0,0) {$\mathscr{R}$};

  \draw (L) -- (P);
  \draw (L) -- (N);
  \draw (P) -- (R);
  \draw (N) -- (R);
\end{tikzpicture}
 \end{center}

 The outcome function $o(G)$ is used to denote the outcome of $G$. The {\em outcome classes} $\L, \N, \R, \P$ are the sets of all positions with the indicated outcome, so that we can write $G\in \L$ when $o(G)=\mathscr{L}$.

 Often, positions decompose into components during the play. In those situations, a player has to choose a component in which to play and the \emph{disjunctive sum} is defined recursively: $$G+H=\{G^\mathcal{L}+H,G+H^\mathcal{L}\,|\,G^\mathcal{R}+H,G+H^\mathcal{R}\}.$$ With the concepts of outcome and disjunctive sum, the relations {\em inequality} and {\em equivalence} of positions are  defined by
$$G \succcurlyeq H  \textrm{ if and only if } o(G+X)\geqslant o(H+X) \textrm{ for all positions } X;$$
$$G\sim H \textrm{ if and only if } o(G+X)= o(H+X) \textrm{ for all positions } X.$$
The first means that replacing $H$ by $G$ can never hurt Left, no matter what the context is; the second means that $G$ acts like $H$ in any context.  In this paper, for ease, we always use the symbol $=$ instead of $\sim$; different situations determine if the symbol is being used for positions or outcomes. Furthermore, $G\succ H$ means that $G \succcurlyeq H$ and $G\not \sim H$. When two positions have exactly the same tree, they are \emph{isomorphic}, denoted as $G\cong H$. Naturally, if $G$ is isomorphic to $H$, $G$ is equal to $H$ in terms of equivalence, but the opposite direction may not be true.

Defining recursively $-G$ as $\{-\GR \! \mid \! -\GL \}$, it is a well-known fact that $G-G=0$, where zero is the terminal position $\{\varnothing\! \mid \! \varnothing \}$ with no options for either player. In fact, combinatorial positions with the disjunctive sum form a partially ordered abelian group, with the terminal position being the identity. In normal play, to determine whether $G$ is larger than or equal to $H$, it suffices to check that the disjunctive sum $G-H \in \L\cup\P$ meaning that, playing second, Left wins $G-H$. This result is crucial, as it is an \emph{options-only} result, without the need to consider the distinguishing positions $X$ as specified in the definition of inequality.

It is also known that each class in the quotient space determined by the equivalence of positions has a simplest and unique representative. This position, whose tree has the fewest number of edges, is referred to as a \emph{canonical form} or a \emph{game value}. In CGT, there are two types of reductions that can be applied to any position: \emph{domination} and \emph{reversibility}. Applying sequentially these reductions in any order to any position allows for the determination of its canonical form. In this document, only domination will be used, as described in the following theorem. It is an intuitive result, reflecting the idea that a player may never chooses an option if they have at least as good an alternative.

\begin{theorem}[Domination]\label{thm:domination}
If $G$ is a position where $G^{L_1}, G^{L_2}\in \GL$ and $G^{L_2}\succcurlyeq G^{L_1}$, then $G=\{\GL\setminus\{G^{L_1}\} \! \mid \! \GR \}$.
\end{theorem}

An important family of game values are \emph{numbers}. In practice, a player only chooses to play on a number in a disjunctive sum if all the components are numbers. In other words, numbers are a kind of ``conquered territory'' that should only be used in endgames, when there is nothing more useful to do. An \emph{integer} corresponds to a guaranteed integer number of moves for a player whose canonical form is $0=\{\varnothing\! \mid \! \varnothing \}$, $n=\{n-1\! \mid \! \varnothing \}$, or $-n=\{\varnothing\! \mid \! -(n-1) \}$, depending on whether it is zero, positive, or negative. A non-integer \emph{dyadic fraction} has the canonical form $\frac{m}{2^j}=\{\frac{m-1}{2^j}\! \mid \!\frac{m+1}{2^j}\}$ or $-\frac{m}{2^j}=\{-\frac{m+1}{2^j}\! \mid \!-\frac{m-1}{2^j}\}$, depending on whether it is positive or negative ($j>0$ and $m$ odd). In these cases as well, there is a practical interpretation: a position $G$ is equal to $\frac{1}{2}$ if it takes two copies of that position in a disjunctive sum to equal one guaranteed move for Left.

If $G^{\mathcal{L}}$ and $G^{\mathcal{R}}$ are all numbers such that all elements of the first are strictly less than all elements of the second, then the position $\{\GL \! \mid \! \GR \}$ is the simplest number between the maximum element of $G^{\mathcal{L}}$ and the minimum element of $G^{\mathcal{R}}$. This fact is known as the \emph{simplicity rule} and can be stated as follows.

\begin{theorem}[Simplicity Rule]\label{thm:simplicity}
If $G=\{G^L \! \mid \! G^R \}$ where $G^L$ and $G^R$ are numbers, and $G^L\prec G^R$, then we have the following:
    \begin{itemize}
        \item[]If there are integers $n$ such that $G^L\prec n\prec G^R$ and $n$ is the smallest in absolute value under these conditions, then $G=n$;
         \item[]If there are no integers $n$ such that $G^L \prec n\prec G^R$, and $d=\frac{i}{2^j}$ is the dyadic fraction between $G^L$ and $G^R$ for which $j$ is minimal, then $G=d$.
    \end{itemize}
\end{theorem}

As we will see shortly, the ruleset {\sc blue-red hackenbush} is closely related to {\sc partisan chocolate game}. Its positions are graphs with blue and red edges that are connected to the ground.
On Left's move, she may remove any blue edge. On Right's move, he may remove any red edge. After either player makes a move, any edge no longer connected to the ground is also removed. The values of alternating strings, as the one illustrated in Figure \ref{fig:fig2}, will be particularly important for what follows and is given by $1-\frac{1}{2}+\frac{1}{4}-\frac{1}{8}+\frac{1}{16}\cdots$

\begin{figure}[htbp]
\begin{center}
\scalebox{0.6}{
\definecolor{ffqqqq}{rgb}{1.,0.,0.}
\definecolor{qqqqff}{rgb}{0.,0.,1.}
\definecolor{ududff}{rgb}{0.30196078431372547,0.30196078431372547,1.}
\definecolor{xdxdff}{rgb}{0.49019607843137253,0.49019607843137253,1.}
\begin{tikzpicture}
\clip(4.,2.5) rectangle (8.,7.5);
\draw [line width=3.6pt] (5.,3.)-- (7.,3.);
\draw [line width=2.pt,color=qqqqff] (6.,3.)-- (6.,4.);
\draw [line width=2.pt,color=ffqqqq] (6.,4.)-- (6.,5.);
\draw [line width=2.pt,color=qqqqff] (6.,5.)-- (6.,6.);
\draw [line width=2.pt,color=ffqqqq] (6.,6.)-- (6.,7.);
\begin{scriptsize}
\draw [fill=white] (6.,3.) circle (2.5pt);
\draw [fill=white] (6.,4.) circle (2.5pt);
\draw [fill=white] (6.,5.) circle (2.5pt);
\draw [fill=white] (6.,6.) circle (2.5pt);
\draw [fill=white] (6.,7.) circle (2.5pt);
\end{scriptsize}
\end{tikzpicture}}
\caption{An alternating {\sc blue-red hackenbush} string with a value of $1-\frac{1}{2}+\frac{1}{4}-\frac{1}{8}=\frac{5}{8}$.}
\label{fig:fig2}
\end{center}
\end{figure}
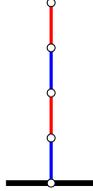

The sequence  explicitly representing the values of the {\sc blue-red hackenbush} alternating strings is $$1,\frac{1}{2},\frac{3}{4},\frac{5}{8},\frac{11}{16},\frac{21}{32},\ldots$$ It is clear that the sequence of denominators is given by $2^{n-1}$. Due to the recursive law $J_{n+1}=2J_n+(-1)^n$,
the sequence of numerators is $1, 1, 3, 5, 11, 21, \ldots$ This is the Jacobsthal Sequence, whose
closed formula is $J_n=\frac{2^n-(-1)^n}{3}$ \cite{OEIS}. Therefore, the values of the {\sc blue-red hackenbush} alternating strings are described by the formula $H_n=\frac{2^n-(-1)^n}{3\times 2^{n-1}}$.

Just as in the case of {\sc blue-red hackenbush}, there are many other rulesets where positions exclusively allow numbers as game values. As we will illustrate, {\sc partisan chocolate game} falls into this category. To prove that, we will rely on a recent finding presented in \cite{Car21}. Consider a set $S$ containing positions from a ruleset. This set is referred to as a \emph{hereditarily closed set of positions of a ruleset} (HCR) if it remains closed when options are taken -- meaning, if an option of an element in $S$ is also an element in $S$. If, for any position $G\in S$, all pairs $(G^L, G^R) \in G^{\mathcal{L}} \times G^{\mathcal{R}}$ satisfy the following property, then all positions $G\in S$ are numbers.

\begin{definition}[F--Loss Property]
Let $S$ be a \emph{HCR}. Given $G\in S$, the pair $(G^L, G^R) \in
G^{\mathcal{L}} \times G^{\mathcal{R}}$ satisfies the F--loss property if there is $G^{RL} \in G^{R{\mathcal{L}}}$ such that $G^{RL} \succcurlyeq G^L$ or there is $G^{LR} \in G^{L{\mathcal{R}}}$ such that $G^{LR} \preccurlyeq G^R$.
\end{definition}

\section{{\sc partisan chocolate game} positions are numbers}
\label{sec:numbers}

A logical first step in analyzing {\sc partisan chocolate game} is the identification of some simple patterns, involving only the previously mentioned options-only comparison. The following theorem states some that are quite useful.

\begin{theorem}\label{thm:patterns}
The comparison of two chocolate bars displays the following patterns.
    \begin{itemize}
        \item[]Given the choice between two chocolate bars with red squares in the upper right corners and the same number of columns (or the same number of rows), Left prefers the larger one. Given the choice between two chocolate bars with blue squares in the upper right corners and the same number of columns (or the same number of rows), Left prefers the smaller one. These facts can be stated as follows.
            \begin{itemize}
            \item[]\emph{(i1a)} If $n$, $m$ and $m'$ are nonnegative integers such that $n+m$ and $n+m'$ are even, and $m\geqslant m'$, then $(n,m)\succcurlyeq (n,m')$. \emph{(i1b)} Analogously, if $n$, $n'$ and $m$ are nonnegative integers such that $n+m$ and $n'+m$ are even and $n\geqslant n'$, then $(n,m)\succcurlyeq (n',m)$.
            \item[]\emph{(i2a)} If $n$, $m$ and $m'$ are nonnegative integers such that $n+m$ and $n+m'$ are odd, and $m\geqslant m'$, then $(n,m)\preccurlyeq (n,m')$. \emph{(i2b)} Analogously, if $n$, $n'$ and $m$ are nonnegative integers such that $n+m$ and $n'+m$ are odd and $n\geqslant n'$, then $(n,m)\preccurlyeq (n',m)$.
            \end{itemize}
         \item[]Given the choice between two chocolate bars with different colored squares in the upper right corners and the same number of columns (or the same number of rows), Left prefers the one with the blue square in the upper right corner. This fact can be stated as follows.
            \begin{itemize}
            \item[] \emph{(iia)} If $n$, $m$ and $m'$ are nonnegative integers such that $n+m$ is odd and $n+m'$ is even, then $(n,m)\succcurlyeq (n,m')$. \emph{(iib)} Analogously, if $n$, $n'$ and $m$ are nonnegative integers such that $n+m$ is odd and $n'+m$ is even, then $(n,m)\succcurlyeq (n',m)$.
            \end{itemize}
            \item[]Two chocolate bars with blue squares in their upper right corners where these corners align on the same downward-sloping diagonal are equivalent. Two chocolate bars with red squares in their upper right corners, aligned on the same downward-sloping diagonal, and having an odd number of columns between them are equivalent. These facts can be stated as follows.
            \begin{itemize}
            \item[]\emph{(iiia)} If $n$, $m$, $n'$ and $m'$ are nonnegative integers such that $n+m=n'+m'$ are odd, then $(n,m)= (n',m')$. \emph{(iiib)} Analogously, if $n$, $m$, $n'$ and $m'$ are nonnegative integers such that $n+m=n'+m'$ are even and $|n-n'|$ is even, then $(n,m)= (n',m')$.
            \end{itemize}
    \end{itemize}
\end{theorem}

Figures \ref{fig:fig3}, \ref{fig:fig4}, and \ref{fig:fig5} illustrate the patterns outlined in Theorem \ref{thm:patterns}. We will see later that some inequalities are strict, but this fact is not needed for this section.

\begin{figure}[htb!]
\begin{center}
\scalebox{0.4}{
\begin{tikzpicture}
\clip(1.8,0.8) rectangle (23.2,6.2);
\fill[line width=0.pt,fill=black,fill opacity=1.0] (3.,2.) -- (2.,2.) -- (2.,1.) -- (3.,1.) -- cycle;
\fill[line width=0.pt,fill=black,fill opacity=1.0] (7.,1.) -- (8.,1.) -- (8.,2.) -- (7.,2.) -- cycle;
\fill[line width=0.pt,fill=black,fill opacity=1.0] (18.,1.) -- (19.,1.) -- (19.,2.) -- (18.,2.) -- cycle;
\fill[line width=0.pt,fill=black,fill opacity=1.0] (14.,1.) -- (14.,2.) -- (13.,2.) -- (13.,1.) -- cycle;
\draw [line width=2.pt] (2.,1.)-- (2.,6.);
\draw [line width=2.pt] (2.,6.)-- (5.,6.);
\draw [line width=2.pt] (2.,1.)-- (5.,1.);
\draw [line width=2.pt] (5.,1.)-- (5.,6.);
\draw [line width=2.pt] (7.,1.)-- (7.,4.);
\draw [line width=2.pt] (7.,4.)-- (10.,4.);
\draw [line width=2.pt] (10.,4.)-- (10.,1.);
\draw [line width=2.pt] (10.,1.)-- (7.,1.);
\draw [line width=2.pt] (2.,2.)-- (5.,2.);
\draw [line width=2.pt] (2.,3.)-- (5.,3.);
\draw [line width=2.pt] (2.,4.)-- (5.,4.);
\draw [line width=2.pt] (2.,5.)-- (5.,5.);
\draw [line width=2.pt] (3.,1.)-- (3.,6.);
\draw [line width=2.pt] (4.,1.)-- (4.,6.);
\draw [line width=2.pt] (7.,2.)-- (10.,2.);
\draw [line width=2.pt] (7.,3.)-- (10.,3.);
\draw [line width=2.pt] (8.,1.)-- (8.,4.);
\draw [line width=2.pt] (9.,1.)-- (9.,4.);
\draw [line width=2.pt] (13.,1.)-- (13.,3.);
\draw [line width=2.pt] (16.,3.)-- (13.,3.);
\draw [line width=2.pt] (13.,1.)-- (16.,1.);
\draw [line width=2.pt] (16.,1.)-- (16.,3.);
\draw [line width=2.pt] (18.,1.)-- (18.,3.);
\draw [line width=2.pt] (18.,3.)-- (23.,3.);
\draw [line width=2.pt] (23.,3.)-- (23.,1.);
\draw [line width=2.pt] (23.,1.)-- (18.,1.);
\draw [line width=2.pt] (14.,1.)-- (14.,3.);
\draw [line width=2.pt] (15.,1.)-- (15.,3.);
\draw [line width=2.pt] (13.,2.)-- (16.,2.);
\draw [line width=2.pt] (18.,2.)-- (23.,2.);
\draw [line width=2.pt] (19.,1.)-- (19.,3.);
\draw [line width=2.pt] (20.,3.)-- (20.,1.);
\draw [line width=2.pt] (21.,1.)-- (21.,3.);
\draw [line width=2.pt] (22.,1.)-- (22.,3.);

\draw (8.1+2-6,2.9+1-1) node[anchor=north west] {\scalebox{1.5}{\textcolor[rgb]{0.00,0.07,1.00}{\textbf{B}}}};
\draw (8.1+2-7,2.9-1) node[anchor=north west] {\scalebox{1.5}{\textcolor[rgb]{0.00,0.07,1.00}{\textbf{B}}}};
\draw (8.1-6,2.9+1-1) node[anchor=north west] {\scalebox{1.5}{\textcolor[rgb]{0.00,0.07,1.00}{\textbf{B}}}};
\draw (8.1+2-7,2.9+1) node[anchor=north west] {\scalebox{1.5}{\textcolor[rgb]{0.00,0.07,1.00}{\textbf{B}}}};
\draw (8.1+2-8,2.9+2) node[anchor=north west] {\scalebox{1.5}{\textcolor[rgb]{0.00,0.07,1.00}{\textbf{B}}}};
\draw (8.1+2-6,2.9+2) node[anchor=north west] {\scalebox{1.5}{\textcolor[rgb]{0.00,0.07,1.00}{\textbf{B}}}};
\draw (8.1+2-7,2.9+3) node[anchor=north west] {\scalebox{1.5}{\textcolor[rgb]{0.00,0.07,1.00}{\textbf{B}}}};

\draw (8.1+2-2,2.9+1) node[anchor=north west] {\scalebox{1.5}{\textcolor[rgb]{0.00,0.07,1.00}{\textbf{B}}}};
\draw (8.1+2-3,2.9) node[anchor=north west] {\scalebox{1.5}{\textcolor[rgb]{0.00,0.07,1.00}{\textbf{B}}}};
\draw (8.1+2-1,2.9) node[anchor=north west] {\scalebox{1.5}{\textcolor[rgb]{0.00,0.07,1.00}{\textbf{B}}}};
\draw (8.1+1-1,2.9-1) node[anchor=north west] {\scalebox{1.5}{\textcolor[rgb]{0.00,0.07,1.00}{\textbf{B}}}};

\draw (8.1+7-1,2.9-1) node[anchor=north west] {\scalebox{1.5}{\textcolor[rgb]{0.00,0.07,1.00}{\textbf{B}}}};
\draw (8.1+8-1,2.9) node[anchor=north west] {\scalebox{1.5}{\textcolor[rgb]{0.00,0.07,1.00}{\textbf{B}}}};
\draw (8.1+6-1,2.9) node[anchor=north west] {\scalebox{1.5}{\textcolor[rgb]{0.00,0.07,1.00}{\textbf{B}}}};

\draw (8.1+11,2.9-1) node[anchor=north west] {\scalebox{1.5}{\textcolor[rgb]{0.00,0.07,1.00}{\textbf{B}}}};
\draw (8.1+13,2.9-1) node[anchor=north west] {\scalebox{1.5}{\textcolor[rgb]{0.00,0.07,1.00}{\textbf{B}}}};
\draw (8.1+14,2.9) node[anchor=north west] {\scalebox{1.5}{\textcolor[rgb]{0.00,0.07,1.00}{\textbf{B}}}};
\draw (8.1+12,2.9) node[anchor=north west] {\scalebox{1.5}{\textcolor[rgb]{0.00,0.07,1.00}{\textbf{B}}}};
\draw (8.1+10,2.9) node[anchor=north west] {\scalebox{1.5}{\textcolor[rgb]{0.00,0.07,1.00}{\textbf{B}}}};

\draw (8.1+2-7,2.9+1-1) node[anchor=north west] {\scalebox{1.5}{\textcolor[rgb]{1.00,0.00,0.00}{\textbf{R}}}};
\draw (8.1+2-6,2.9+1-2) node[anchor=north west] {\scalebox{1.5}{\textcolor[rgb]{1.00,0.00,0.00}{\textbf{R}}}};
\draw (8.1+2-7-1,2.9+1) node[anchor=north west] {\scalebox{1.5}{\textcolor[rgb]{1.00,0.00,0.00}{\textbf{R}}}};
\draw (8.1+2-7+1,2.9+1) node[anchor=north west] {\scalebox{1.5}{\textcolor[rgb]{1.00,0.00,0.00}{\textbf{R}}}};
\draw (8.1+2-7,2.9+1+1) node[anchor=north west] {\scalebox{1.5}{\textcolor[rgb]{1.00,0.00,0.00}{\textbf{R}}}};
\draw (8.1+2-8,2.9+1+2) node[anchor=north west] {\scalebox{1.5}{\textcolor[rgb]{1.00,0.00,0.00}{\textbf{R}}}};
\draw (8.1+2-6,2.9+1+2) node[anchor=north west] {\scalebox{1.5}{\textcolor[rgb]{1.00,0.00,0.00}{\textbf{R}}}};

\draw (8.1+2-2,2.9+1-1) node[anchor=north west] {\scalebox{1.5}{\textcolor[rgb]{1.00,0.00,0.00}{\textbf{R}}}};
\draw (8.1+2-1,2.9-1) node[anchor=north west] {\scalebox{1.5}{\textcolor[rgb]{1.00,0.00,0.00}{\textbf{R}}}};
\draw (8.1+2-1,2.9+1) node[anchor=north west] {\scalebox{1.5}{\textcolor[rgb]{1.00,0.00,0.00}{\textbf{R}}}};
\draw (8.1-1,2.9+1) node[anchor=north west] {\scalebox{1.5}{\textcolor[rgb]{1.00,0.00,0.00}{\textbf{R}}}};

\draw (8.1+6,2.9) node[anchor=north west] {\scalebox{1.5}{\textcolor[rgb]{1.00,0.00,0.00}{\textbf{R}}}};
\draw (8.1+7,2.9-1) node[anchor=north west] {\scalebox{1.5}{\textcolor[rgb]{1.00,0.00,0.00}{\textbf{R}}}};

\draw (8.1+7+5,2.9-1) node[anchor=north west] {\scalebox{1.5}{\textcolor[rgb]{1.00,0.00,0.00}{\textbf{R}}}};
\draw (8.1+7+7,2.9-1) node[anchor=north west] {\scalebox{1.5}{\textcolor[rgb]{1.00,0.00,0.00}{\textbf{R}}}};
\draw (8.1+7+6,2.9) node[anchor=north west] {\scalebox{1.5}{\textcolor[rgb]{1.00,0.00,0.00}{\textbf{R}}}};
\draw (8.1+7+4,2.9) node[anchor=north west] {\scalebox{1.5}{\textcolor[rgb]{1.00,0.00,0.00}{\textbf{R}}}};

\draw (5.5,2.6) node[anchor=north west] {\scalebox{2.5}{$\succcurlyeq$}};
\draw (5.5+11,2.6) node[anchor=north west] {\scalebox{2.5}{$\succcurlyeq$}};
\end{tikzpicture}}
\caption{Items i1a and i2b of Theorem \ref{thm:patterns}.}
\label{fig:fig3}
\end{center}
\end{figure}
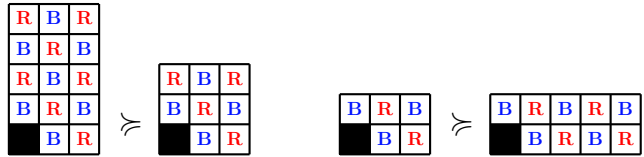

\begin{figure}[htb!]
\begin{center}
\scalebox{0.4}{
\begin{tikzpicture}
\clip(1.8,0.8) rectangle (10.2,6.2);
\fill[line width=0.pt,fill=black,fill opacity=1.0] (3.,2.) -- (2.,2.) -- (2.,1.) -- (3.,1.) -- cycle;
\fill[line width=0.pt,fill=black,fill opacity=1.0] (7.,1.) -- (8.,1.) -- (8.,2.) -- (7.,2.) -- cycle;
\draw [line width=2.pt] (2.,1.)-- (2.,6.);
\draw [line width=2.pt] (2.,6.)-- (5.,6.);
\draw [line width=2.pt] (2.,1.)-- (5.,1.);
\draw [line width=2.pt] (5.,1.)-- (5.,6.);
\draw [line width=2.pt] (7.,1.)-- (7.,3.);
\draw [line width=2.pt] (10.,3.)-- (10.,1.);
\draw [line width=2.pt] (10.,1.)-- (7.,1.);
\draw [line width=2.pt] (2.,2.)-- (5.,2.);
\draw [line width=2.pt] (2.,3.)-- (5.,3.);
\draw [line width=2.pt] (2.,4.)-- (5.,4.);
\draw [line width=2.pt] (2.,5.)-- (5.,5.);
\draw [line width=2.pt] (3.,1.)-- (3.,6.);
\draw [line width=2.pt] (4.,1.)-- (4.,6.);
\draw [line width=2.pt] (7.,2.)-- (10.,2.);
\draw [line width=2.pt] (7.,3.)-- (10.,3.);
\draw [line width=2.pt] (8.,1.)-- (8.,3.);
\draw [line width=2.pt] (9.,1.)-- (9.,3.);

\draw (8.1+2-6,2.9+1-1) node[anchor=north west] {\scalebox{1.5}{\textcolor[rgb]{0.00,0.07,1.00}{\textbf{B}}}};
\draw (8.1+2-7,2.9-1) node[anchor=north west] {\scalebox{1.5}{\textcolor[rgb]{0.00,0.07,1.00}{\textbf{B}}}};
\draw (8.1-6,2.9+1-1) node[anchor=north west] {\scalebox{1.5}{\textcolor[rgb]{0.00,0.07,1.00}{\textbf{B}}}};
\draw (8.1+2-7,2.9+1) node[anchor=north west] {\scalebox{1.5}{\textcolor[rgb]{0.00,0.07,1.00}{\textbf{B}}}};
\draw (8.1+2-8,2.9+2) node[anchor=north west] {\scalebox{1.5}{\textcolor[rgb]{0.00,0.07,1.00}{\textbf{B}}}};
\draw (8.1+2-6,2.9+2) node[anchor=north west] {\scalebox{1.5}{\textcolor[rgb]{0.00,0.07,1.00}{\textbf{B}}}};
\draw (8.1+2-7,2.9+3) node[anchor=north west] {\scalebox{1.5}{\textcolor[rgb]{0.00,0.07,1.00}{\textbf{B}}}};

\draw (8.1+2-3,2.9) node[anchor=north west] {\scalebox{1.5}{\textcolor[rgb]{0.00,0.07,1.00}{\textbf{B}}}};
\draw (8.1+2-1,2.9) node[anchor=north west] {\scalebox{1.5}{\textcolor[rgb]{0.00,0.07,1.00}{\textbf{B}}}};
\draw (8.1+1-1,2.9-1) node[anchor=north west] {\scalebox{1.5}{\textcolor[rgb]{0.00,0.07,1.00}{\textbf{B}}}};

\draw (8.1+2-7,2.9+1-1) node[anchor=north west] {\scalebox{1.5}{\textcolor[rgb]{1.00,0.00,0.00}{\textbf{R}}}};
\draw (8.1+2-6,2.9+1-2) node[anchor=north west] {\scalebox{1.5}{\textcolor[rgb]{1.00,0.00,0.00}{\textbf{R}}}};
\draw (8.1+2-7-1,2.9+1) node[anchor=north west] {\scalebox{1.5}{\textcolor[rgb]{1.00,0.00,0.00}{\textbf{R}}}};
\draw (8.1+2-7+1,2.9+1) node[anchor=north west] {\scalebox{1.5}{\textcolor[rgb]{1.00,0.00,0.00}{\textbf{R}}}};
\draw (8.1+2-7,2.9+1+1) node[anchor=north west] {\scalebox{1.5}{\textcolor[rgb]{1.00,0.00,0.00}{\textbf{R}}}};
\draw (8.1+2-8,2.9+1+2) node[anchor=north west] {\scalebox{1.5}{\textcolor[rgb]{1.00,0.00,0.00}{\textbf{R}}}};
\draw (8.1+2-6,2.9+1+2) node[anchor=north west] {\scalebox{1.5}{\textcolor[rgb]{1.00,0.00,0.00}{\textbf{R}}}};

\draw (8.1+2-2,2.9+1-1) node[anchor=north west] {\scalebox{1.5}{\textcolor[rgb]{1.00,0.00,0.00}{\textbf{R}}}};
\draw (8.1+2-1,2.9-1) node[anchor=north west] {\scalebox{1.5}{\textcolor[rgb]{1.00,0.00,0.00}{\textbf{R}}}};

\draw (5.5,2.6) node[anchor=north west] {\scalebox{2.5}{$\preccurlyeq$}};
\end{tikzpicture}}
\caption{Item iia of Theorem \ref{thm:patterns}.}
\label{fig:fig4}
\end{center}
\end{figure}
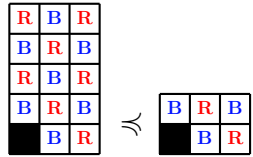

\clearpage
\begin{figure}[htb!]
\begin{center}
\scalebox{0.4}{
\begin{tikzpicture}
\clip(1.8,0.8) rectangle (19.2,9.2);
\draw (8.1+2-6,2.9+1-1) node[anchor=north west] {\scalebox{1.5}{\textcolor[rgb]{0.00,0.07,1.00}{\textbf{B}}}};
\draw (8.1-6,2.9+1-1) node[anchor=north west] {\scalebox{1.5}{\textcolor[rgb]{0.00,0.07,1.00}{\textbf{B}}}};
\draw (8.1+2-4,2.9+1-1) node[anchor=north west] {\scalebox{1.5}{\textcolor[rgb]{0.00,0.07,1.00}{\textbf{B}}}};
\draw (8.1+2-2,2.9+1-1) node[anchor=north west] {\scalebox{1.5}{\textcolor[rgb]{0.00,0.07,1.00}{\textbf{B}}}};
\draw (8.1+1-2,2.9-1) node[anchor=north west] {\scalebox{1.5}{\textcolor[rgb]{0.00,0.07,1.00}{\textbf{B}}}};
\draw (8.1-1-2,2.9-1) node[anchor=north west] {\scalebox{1.5}{\textcolor[rgb]{0.00,0.07,1.00}{\textbf{B}}}};
\draw (8.1-3-2,2.9-1) node[anchor=north west] {\scalebox{1.5}{\textcolor[rgb]{0.00,0.07,1.00}{\textbf{B}}}};
\draw (8.1+2-7,2.9+1) node[anchor=north west] {\scalebox{1.5}{\textcolor[rgb]{0.00,0.07,1.00}{\textbf{B}}}};
\draw (8.1+2-5,2.9+1) node[anchor=north west] {\scalebox{1.5}{\textcolor[rgb]{0.00,0.07,1.00}{\textbf{B}}}};
\draw (8.1+1-5,2.9+2) node[anchor=north west] {\scalebox{1.5}{\textcolor[rgb]{0.00,0.07,1.00}{\textbf{B}}}};
\draw (8.1+1-7,2.9+2) node[anchor=north west] {\scalebox{1.5}{\textcolor[rgb]{0.00,0.07,1.00}{\textbf{B}}}};
\draw (8.1+1-6,2.9+3) node[anchor=north west] {\scalebox{1.5}{\textcolor[rgb]{0.00,0.07,1.00}{\textbf{B}}}};
\draw (8.1+3-6,2.9+3) node[anchor=north west] {\scalebox{1.5}{\textcolor[rgb]{0.00,0.07,1.00}{\textbf{B}}}};
\draw (8.1+6-1,2.9+3) node[anchor=north west] {\scalebox{1.5}{\textcolor[rgb]{0.00,0.07,1.00}{\textbf{B}}}};
\draw (8.1+5-1,2.9+2) node[anchor=north west] {\scalebox{1.5}{\textcolor[rgb]{0.00,0.07,1.00}{\textbf{B}}}};
\draw (8.1+7-1,2.9+2) node[anchor=north west] {\scalebox{1.5}{\textcolor[rgb]{0.00,0.07,1.00}{\textbf{B}}}};
\draw (8.1+6-1,2.9+1) node[anchor=north west] {\scalebox{1.5}{\textcolor[rgb]{0.00,0.07,1.00}{\textbf{B}}}};
\draw (8.1+8-1,2.9+1) node[anchor=north west] {\scalebox{1.5}{\textcolor[rgb]{0.00,0.07,1.00}{\textbf{B}}}};
\draw (8.1+9-1,2.9) node[anchor=north west] {\scalebox{1.5}{\textcolor[rgb]{0.00,0.07,1.00}{\textbf{B}}}};
\draw (8.1+7-1,2.9) node[anchor=north west] {\scalebox{1.5}{\textcolor[rgb]{0.00,0.07,1.00}{\textbf{B}}}};
\draw (8.1+5-1,2.9) node[anchor=north west] {\scalebox{1.5}{\textcolor[rgb]{0.00,0.07,1.00}{\textbf{B}}}};
\draw (8.1+6-1,2.9-1) node[anchor=north west] {\scalebox{1.5}{\textcolor[rgb]{0.00,0.07,1.00}{\textbf{B}}}};
\draw (8.1+8-1,2.9-1) node[anchor=north west] {\scalebox{1.5}{\textcolor[rgb]{0.00,0.07,1.00}{\textbf{B}}}};

\draw (8.1+3-6,2.9+1-1) node[anchor=north west]  {\scalebox{1.5}{\textcolor[rgb]{1.00,0.00,0.00}{\textbf{R}}}};
\draw (8.1-6+1,2.9+1-1) node[anchor=north west]  {\scalebox{1.5}{\textcolor[rgb]{1.00,0.00,0.00}{\textbf{R}}}};
\draw (8.1+3-4,2.9+1-1) node[anchor=north west]  {\scalebox{1.5}{\textcolor[rgb]{1.00,0.00,0.00}{\textbf{R}}}};
\draw (8.1-2,2.9-1) node[anchor=north west] {\scalebox{1.5}{\textcolor[rgb]{1.00,0.00,0.00}{\textbf{R}}}};
\draw (8.1-4,2.9-1) node[anchor=north west] {\scalebox{1.5}{\textcolor[rgb]{1.00,0.00,0.00}{\textbf{R}}}};
\draw (8.1,2.9-1) node[anchor=north west] {\scalebox{1.5}{\textcolor[rgb]{1.00,0.00,0.00}{\textbf{R}}}};
\draw (8.1+3-7,2.9+1) node[anchor=north west]  {\scalebox{1.5}{\textcolor[rgb]{1.00,0.00,0.00}{\textbf{R}}}};
\draw (8.1+1-7,2.9+1) node[anchor=north west]  {\scalebox{1.5}{\textcolor[rgb]{1.00,0.00,0.00}{\textbf{R}}}};
\draw (8.1+1-6,2.9+2) node[anchor=north west]  {\scalebox{1.5}{\textcolor[rgb]{1.00,0.00,0.00}{\textbf{R}}}};
\draw (8.1+1-4,2.9+2) node[anchor=north west]  {\scalebox{1.5}{\textcolor[rgb]{1.00,0.00,0.00}{\textbf{R}}}};
\draw (8.1-4,2.9+3) node[anchor=north west]  {\scalebox{1.5}{\textcolor[rgb]{1.00,0.00,0.00}{\textbf{R}}}};
\draw (8.1-6,2.9+3) node[anchor=north west]  {\scalebox{1.5}{\textcolor[rgb]{1.00,0.00,0.00}{\textbf{R}}}};
\draw (8.1+4,2.9+3) node[anchor=north west]  {\scalebox{1.5}{\textcolor[rgb]{1.00,0.00,0.00}{\textbf{R}}}};
\draw (8.1+6,2.9+3) node[anchor=north west]  {\scalebox{1.5}{\textcolor[rgb]{1.00,0.00,0.00}{\textbf{R}}}};
\draw (8.1+5,2.9+2) node[anchor=north west]  {\scalebox{1.5}{\textcolor[rgb]{1.00,0.00,0.00}{\textbf{R}}}};
\draw (8.1+4,2.9+1) node[anchor=north west]  {\scalebox{1.5}{\textcolor[rgb]{1.00,0.00,0.00}{\textbf{R}}}};
\draw (8.1+6,2.9+1) node[anchor=north west]  {\scalebox{1.5}{\textcolor[rgb]{1.00,0.00,0.00}{\textbf{R}}}};
\draw (8.1+8,2.9+1) node[anchor=north west]  {\scalebox{1.5}{\textcolor[rgb]{1.00,0.00,0.00}{\textbf{R}}}};
\draw (8.1+5,2.9) node[anchor=north west]  {\scalebox{1.5}{\textcolor[rgb]{1.00,0.00,0.00}{\textbf{R}}}};
\draw (8.1+7,2.9) node[anchor=north west]  {\scalebox{1.5}{\textcolor[rgb]{1.00,0.00,0.00}{\textbf{R}}}};
\draw (8.1+6,2.9-1) node[anchor=north west]  {\scalebox{1.5}{\textcolor[rgb]{1.00,0.00,0.00}{\textbf{R}}}};
\draw (8.1+8,2.9-1) node[anchor=north west]  {\scalebox{1.5}{\textcolor[rgb]{1.00,0.00,0.00}{\textbf{R}}}};

\fill[line width=0.pt,fill=black,fill opacity=1.0] (3.,2.) -- (2.,2.) -- (2.,1.) -- (3.,1.) -- cycle;
\fill[line width=0.pt,fill=black,fill opacity=1.0] (13.,2.) -- (12.,2.) -- (12.,1.) -- (13.,1.) -- cycle;
\draw [line width=4pt] (2.,1.)-- (2.,6.);
\draw [line width=4pt] (2.,6.)-- (5.,6.);
\draw [line width=4pt] (2.,1.)-- (5.,1.);
\draw [line width=2.pt] (5.,1.)-- (5.,6.);
\draw [line width=2.pt] (2.,2.)-- (5.,2.);
\draw [line width=4pt] (2.,3.)-- (5.,3.);
\draw [line width=2.pt] (2.,4.)-- (5.,4.);
\draw [line width=2.pt] (2.,5.)-- (5.,5.);
\draw [line width=2.pt] (3.,1.)-- (3.,6.);
\draw [line width=2.pt] (4.,1.)-- (4.,6.);
\draw [line width=4pt] (5.,6.)-- (6.,6.);
\draw [line width=4pt] (6.,6.)-- (6.,1.);
\draw [line width=4pt] (6.,1.)-- (5.,1.);
\draw [line width=2.pt] (5.,2.)-- (6.,2.);
\draw [line width=4pt] (5.,3.)-- (6.,3.);
\draw [line width=2.pt] (5.,4.)-- (6.,4.);
\draw [line width=2.pt] (5.,5.)-- (6.,5.);
\draw [line width=4pt] (9.,3.)-- (9.,1.);
\draw [line width=4pt] (9.,3.)-- (6.,3.);
\draw [line width=4pt] (6.,1.)-- (9.,1.);
\draw [line width=2.pt] (9.,2.)-- (6.,2.);
\draw [line width=2.pt] (7.,3.)-- (7.,1.);
\draw [line width=2.pt] (8.,3.)-- (8.,1.);
\draw [line width=2.pt,dash pattern=on 3pt off 3pt] (2.,9.)-- (10.,1.);
\draw [line width=4pt] (12.,1.)-- (12.,6.);
\draw [line width=4pt] (12.,6.)-- (15.,6.);
\draw [line width=4pt] (15.,6.)-- (15.,1.);
\draw [line width=4pt] (12.,1.)-- (17.,1.);
\draw [line width=4pt] (17.,1.)-- (17.,4.);
\draw [line width=4pt] (17.,4.)-- (12.,4.);
\draw [line width=2.pt] (13.,1.)-- (13.,6.);
\draw [line width=2.pt] (14.,1.)-- (14.,6.);
\draw [line width=2.pt] (17.,3.)-- (12.,3.);
\draw [line width=2.pt] (12.,2.)-- (17.,2.);
\draw [line width=2.pt] (12.,5.)-- (15.,5.);
\draw [line width=2.pt] (16.,4.)-- (16.,1.);
\draw [line width=2.pt,dash pattern=on 3pt off 3pt] (12.,8.)-- (19.,1.);
\end{tikzpicture}}
\caption{Items iiia and iiib of Theorem \ref{thm:patterns}. On both the left and the right, there are two chocolates overlapping.}
\label{fig:fig5}
\end{center}
\end{figure}
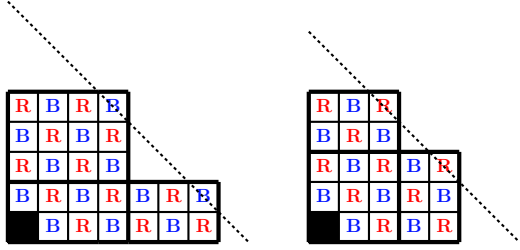

\begin{proof}[Proof of Theorem \ref{thm:patterns}]
We will prove only the first case. Although tedious, the proofs of the various cases follow the traditional mathematical induction, all of them being totally similar and not particularly difficult.

Let $n$, $m$ and $m'$ be nonnegative integers such that $n+m$ and $n+m'$ are even, and $m\geqslant m'$. Let us analyze the disjunctive sum $(n,m)-(n,m')$. If Right moves first to $(n,m)-(n,m'-2k)$, Left answers with the option $(n,m'-2k)-(n,m'-2k)=0$ and wins. If Right moves to $(n,m)-(n-2k,m')$, Left answers with  $(n-2k,m)-(n-2k,m')$ and, by induction (i1a), Left wins. If Right moves to $(n,m-(2k + 1))-(n, m')$, and $m-(2k+1)> m'$,\linebreak Left answers with $(n, m')-(n, m')=0$ and wins. If Right moves first to $(n,m-(2k + 1))-(n, m')$, and $m-(2k+1)<m'$, Left answers with $(n,m-(2k + 1))-(n,m-(2k + 1))=0$ and wins. If Right moves to $(n-(2k+1),m)-(n,m')$, and $m-(2k+1)> m'$, Left answers with the option $(n-(2k+2),m)-(n,m')$. In that case, by induction (iiib), $(n-(2k+2),m)=(n,m-(2k+2))$, and, so, $(n-(2k+2),m)\succcurlyeq (n,m')$ (i1a), and Left wins. If Right moves to $(n-(2k+1),m)-(n,m')$, and $m'> m-(2k+1)\geqslant 0$, Left answers with $(n-(2k+1),m)-(n,m-(2k+1))$. By induction (iiia), this sum is equal to zero, and Left wins. If Right moves to $(n-(2k+1),m)-(n,m')$, and $m'\geqslant 0>m-(2k+1)$, Left answers with the option $(n-(2k+1),m)-(n-m'+m-(2k+1),m')$. By induction (iiia), this sum is equal to zero and Left wins. In all cases, Left wins playing second. Hence, $(n,m)\succcurlyeq (n,m')$.
\end{proof}

Given a {\sc partisan chocolate game} position, we say that a pair of Left and Right options is \emph{concordant} if both are made by vertical cuts or if both are made by horizontal cuts. If that is not the case, one player makes a vertical cut while the other makes a horizontal cut, and the pair is \emph{discordant}.

Returning to the $(2,3)$--position shown in Figure \ref{fig:fig1}, the pair of options  $G^{L_1}$ and $G^{R_1}$ are concordant. It is easy to see that, in this case, one of the options can be reached from the other. In other words, a concordant pair trivially satisfies the F-loss Property. Figure \ref{fig:fig6} illustrates this fact.

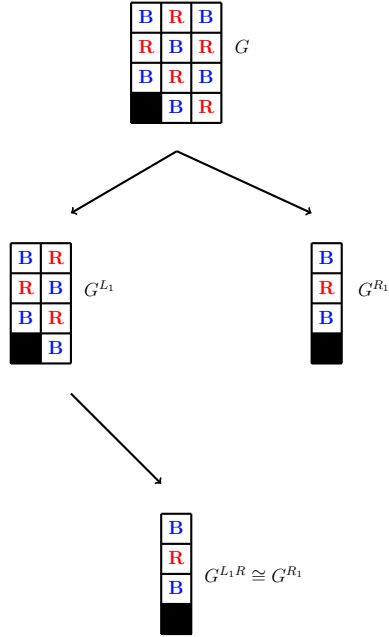
\begin{figure}[htb!]
\begin{center}
\scalebox{0.4}{
\begin{tikzpicture}
\clip(-0.5,-8.5) rectangle (13.2,13.5);
\fill[line width=0.pt,fill=black,fill opacity=1.0] (4.,9.) -- (4.,10.) -- (5.,10.) -- (5.,9.) -- cycle;
\fill[line width=0.pt,fill=black,fill opacity=1.0] (0.,1.) -- (0.,2.) -- (1.,2.) -- (1.,1.) -- cycle;
\fill[line width=0.pt,fill=black,fill opacity=1.0] (5.,-8.) -- (5.,-7.) -- (6.,-7.) -- (6.,-8.) -- cycle;
\fill[line width=0.pt,fill=black,fill opacity=1.0] (10.,1.) -- (10.,2.) -- (11.,2.) -- (11.,1.) -- cycle;
\draw [line width=2.pt] (4.,9.)-- (4.,13.);
\draw [line width=2.pt] (4.,13.)-- (7.,13.);
\draw [line width=2.pt] (4.,9.)-- (7.,9.);
\draw [line width=2.pt] (7.,9.)-- (7.,13.);
\draw [line width=2.pt] (4.,12.)-- (7.,12.);
\draw [line width=2.pt] (4.,11.)-- (7.,11.);
\draw [line width=2.pt] (4.,10.)-- (7.,10.);
\draw [line width=2.pt] (5.,13.)-- (5.,9.);
\draw [line width=2.pt] (6.,13.)-- (6.,9.);
\draw [->,line width=2.pt] (5.52,8.06) -- (2.,6.);
\draw [->,line width=2.pt] (5.52,8.06) -- (10.,6.);
\draw [line width=2.pt] (0.,1.)-- (0.,5.);
\draw [line width=2.pt] (1.,5.)-- (1.,1.);
\draw [line width=2.pt] (2.,5.)-- (2.,1.);
\draw [line width=2.pt] (0.,5.)-- (2.,5.);
\draw [line width=2.pt] (0.,4.)-- (2.,4.);
\draw [line width=2.pt] (0.,3.)-- (2.,3.);
\draw [line width=2.pt] (0.,2.)-- (2.,2.);
\draw [line width=2.pt] (0.,1.)-- (2.,1.);
\draw [line width=2.pt] (10.,5.)-- (10.,1.);
\draw [line width=2.pt] (10.,1.)-- (11.,1.);
\draw [line width=2.pt] (11.,1.)-- (11.,5.);
\draw [line width=2.pt] (11.,5.)-- (10.,5.);
\draw [line width=2.pt] (10.,4.)-- (11.,4.);
\draw [line width=2.pt] (10.,3.)-- (11.,3.);
\draw [line width=2.pt] (10.,2.)-- (11.,2.);
\draw [->,line width=2.pt] (2.,0.) -- (5.,-3.);
\draw [line width=2.pt] (5.,-4.)-- (6.,-4.);
\draw [line width=2.pt] (5.,-4.)-- (5.,-8.);
\draw [line width=2.pt] (5.,-8.)-- (6.,-8.);
\draw [line width=2.pt] (6.,-8.)-- (6.,-4.);
\draw [line width=2.pt] (5.,-5.)-- (6.,-5.);
\draw [line width=2.pt] (5.,-6.)-- (6.,-6.);
\draw [line width=2.pt] (5.,-7.)-- (6.,-7.);

\draw (8.1-0.8,2.9+9) node[anchor=north west] {\scalebox{1.5}{$G$}};
\draw (8.1-5.8,2.9+1) node[anchor=north west] {\scalebox{1.5}{$G^{L_1}$}};
\draw (8.1+3.3,2.9+1) node[anchor=north west] {\scalebox{1.5}{$G^{R_1}$}};
\draw (8.1-1.8,2.9-8.5) node[anchor=north west] {\scalebox{1.5}{$G^{L_1R}\cong G^{R_1}$}};

\draw (8.1-8,2.9+2) node[anchor=north west] {\scalebox{1.5}{\textcolor[rgb]{0.00,0.07,1.00}{\textbf{B}}}};
\draw (8.1-7,2.9+1) node[anchor=north west] {\scalebox{1.5}{\textcolor[rgb]{0.00,0.07,1.00}{\textbf{B}}}};
\draw (8.1-8,2.9) node[anchor=north west] {\scalebox{1.5}{\textcolor[rgb]{0.00,0.07,1.00}{\textbf{B}}}};
\draw (8.1-7,2.9-1) node[anchor=north west] {\scalebox{1.5}{\textcolor[rgb]{0.00,0.07,1.00}{\textbf{B}}}};
\draw (8.1+2,2.9+2) node[anchor=north west] {\scalebox{1.5}{\textcolor[rgb]{0.00,0.07,1.00}{\textbf{B}}}};
\draw (8.1+2,2.9) node[anchor=north west] {\scalebox{1.5}{\textcolor[rgb]{0.00,0.07,1.00}{\textbf{B}}}};
\draw (8.1-3,2.9-7) node[anchor=north west] {\scalebox{1.5}{\textcolor[rgb]{0.00,0.07,1.00}{\textbf{B}}}};
\draw (8.1-3,2.9-9) node[anchor=north west] {\scalebox{1.5}{\textcolor[rgb]{0.00,0.07,1.00}{\textbf{B}}}};
\draw (8.1-4,2.9+10) node[anchor=north west] {\scalebox{1.5}{\textcolor[rgb]{0.00,0.07,1.00}{\textbf{B}}}};
\draw (8.1-4+2,2.9+10) node[anchor=north west] {\scalebox{1.5}{\textcolor[rgb]{0.00,0.07,1.00}{\textbf{B}}}};
\draw (8.1-5+2,2.9+9) node[anchor=north west] {\scalebox{1.5}{\textcolor[rgb]{0.00,0.07,1.00}{\textbf{B}}}};
\draw (8.1-6+2,2.9+8) node[anchor=north west] {\scalebox{1.5}{\textcolor[rgb]{0.00,0.07,1.00}{\textbf{B}}}};
\draw (8.1-4+2,2.9+8) node[anchor=north west] {\scalebox{1.5}{\textcolor[rgb]{0.00,0.07,1.00}{\textbf{B}}}};
\draw (8.1-5+2,2.9+7) node[anchor=north west] {\scalebox{1.5}{\textcolor[rgb]{0.00,0.07,1.00}{\textbf{B}}}};

\draw (8.1+2,2.9+1) node[anchor=north west] {\scalebox{1.5}{\textcolor[rgb]{1.00,0.00,0.00}{\textbf{R}}}};
\draw (8.1-8,2.9+1) node[anchor=north west] {\scalebox{1.5}{\textcolor[rgb]{1.00,0.00,0.00}{\textbf{R}}}};
\draw (8.1-7,2.9+2) node[anchor=north west] {\scalebox{1.5}{\textcolor[rgb]{1.00,0.00,0.00}{\textbf{R}}}};
\draw (8.1-7,2.9+2-2) node[anchor=north west] {\scalebox{1.5}{\textcolor[rgb]{1.00,0.00,0.00}{\textbf{R}}}};
\draw (8.1-3,2.9+2-10) node[anchor=north west] {\scalebox{1.5}{\textcolor[rgb]{1.00,0.00,0.00}{\textbf{R}}}};
\draw (8.1-3+1,2.9+7) node[anchor=north west] {\scalebox{1.5}{\textcolor[rgb]{1.00,0.00,0.00}{\textbf{R}}}};
\draw (8.1-3,2.9+8) node[anchor=north west] {\scalebox{1.5}{\textcolor[rgb]{1.00,0.00,0.00}{\textbf{R}}}};
\draw (8.1-4,2.9+9) node[anchor=north west] {\scalebox{1.5}{\textcolor[rgb]{1.00,0.00,0.00}{\textbf{R}}}};
\draw (8.1-4+2,2.9+9) node[anchor=north west] {\scalebox{1.5}{\textcolor[rgb]{1.00,0.00,0.00}{\textbf{R}}}};
\draw (8.1-4+1,2.9+10) node[anchor=north west] {\scalebox{1.5}{\textcolor[rgb]{1.00,0.00,0.00}{\textbf{R}}}};

\end{tikzpicture}}
\caption{Concordant options trivially satisfy the F--loss Property.}
\label{fig:fig6}
\end{center}
\end{figure}

On the other hand, the pair of options $G^{L_1}$ and $G^{R_2}$ are discordant. In this case, the fact that the F--loss property holds is no longer as straightforward. However, it is possible to observe that one of the players can find a cut in the chocolate bar that makes use of Theorem \ref{thm:patterns}. Figure \ref{fig:fig7} illustrates this fact. The precise method for doing so in all cases is described in the proof of the following theorem. 

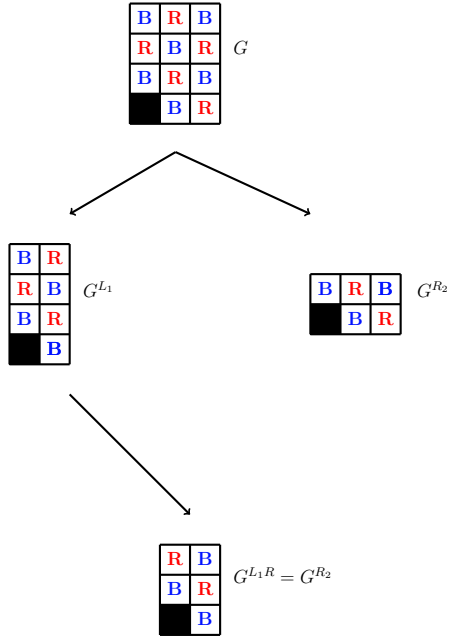
\begin{figure}[htb!]
\begin{center}
\scalebox{0.4}{
\begin{tikzpicture}
\clip(-0.5,-8.5) rectangle (15.2,13.5);
\fill[line width=0.pt,fill=black,fill opacity=1.0] (4.,9.) -- (4.,10.) -- (5.,10.) -- (5.,9.) -- cycle;
\fill[line width=0.pt,fill=black,fill opacity=1.0] (0.,1.) -- (0.,2.) -- (1.,2.) -- (1.,1.) -- cycle;
\fill[line width=0.pt,fill=black,fill opacity=1.0] (5.,-8.) -- (5.,-7.) -- (6.,-7.) -- (6.,-8.) -- cycle;
\fill[line width=0.pt,fill=black,fill opacity=1.0] (10.,1.+1) -- (10.,2.+1) -- (11.,2.+1) -- (11.,1.+1) -- cycle;
\draw [line width=2.pt] (4.,9.)-- (4.,13.);
\draw [line width=2.pt] (4.,13.)-- (7.,13.);
\draw [line width=2.pt] (4.,9.)-- (7.,9.);
\draw [line width=2.pt] (7.,9.)-- (7.,13.);
\draw [line width=2.pt] (4.,12.)-- (7.,12.);
\draw [line width=2.pt] (4.,11.)-- (7.,11.);
\draw [line width=2.pt] (4.,10.)-- (7.,10.);
\draw [line width=2.pt] (5.,13.)-- (5.,9.);
\draw [line width=2.pt] (6.,13.)-- (6.,9.);
\draw [->,line width=2.pt] (5.52,8.06) -- (2.,6.);
\draw [->,line width=2.pt] (5.52,8.06) -- (10.,6.);
\draw [line width=2.pt] (0.,1.)-- (0.,5.);
\draw [line width=2.pt] (1.,5.)-- (1.,1.);
\draw [line width=2.pt] (2.,5.)-- (2.,1.);
\draw [line width=2.pt] (0.,5.)-- (2.,5.);
\draw [line width=2.pt] (0.,4.)-- (2.,4.);
\draw [line width=2.pt] (0.,3.)-- (2.,3.);
\draw [line width=2.pt] (0.,2.)-- (2.,2.);
\draw [line width=2.pt] (0.,1.)-- (2.,1.);
\draw [line width=2.pt] (10.,4.)-- (10.,2.);
\draw [line width=2.pt] (11.,2.)-- (11.,4.);
\draw [line width=2.pt] (12.,2.)-- (12.,4.);
\draw [line width=2.pt] (13.,2.)-- (13.,4.);
\draw [line width=2.pt] (10.,4.)-- (13.,4.);
\draw [line width=2.pt] (10.,3.)-- (13.,3.);
\draw [line width=2.pt] (10.,2.)-- (13.,2.);
\draw [->,line width=2.pt] (2.,0.) -- (6.,-4.);
\draw [line width=2.pt] (5.,-5.)-- (5.,-8.);
\draw [line width=2.pt] (5.,-8.)-- (7.,-8.);
\draw [line width=2.pt] (6.,-8.)-- (6.,-5.);
\draw [line width=2.pt] (5.,-5.)-- (7.,-5.);
\draw [line width=2.pt] (5.,-6.)-- (7.,-6.);
\draw [line width=2.pt] (5.,-7.)-- (7.,-7.);
\draw [line width=2.pt] (7.,-8.)-- (7.,-5.);

\draw (8.1-0.8,2.9+9) node[anchor=north west] {\scalebox{1.5}{$G$}};
\draw (8.1-5.8,2.9+1) node[anchor=north west] {\scalebox{1.5}{$G^{L_1}$}};
\draw (8.1+5.3,2.9+1) node[anchor=north west] {\scalebox{1.5}{$G^{R_2}$}};
\draw (8.1-0.8,2.9-8.5) node[anchor=north west] {\scalebox{1.5}{$G^{L_1R}= G^{R_2}$}};

\draw (8.1-8,2.9+2) node[anchor=north west] {\scalebox{1.5}{\textcolor[rgb]{0.00,0.07,1.00}{\textbf{B}}}};
\draw (8.1-7,2.9+1) node[anchor=north west] {\scalebox{1.5}{\textcolor[rgb]{0.00,0.07,1.00}{\textbf{B}}}};
\draw (8.1-8,2.9) node[anchor=north west] {\scalebox{1.5}{\textcolor[rgb]{0.00,0.07,1.00}{\textbf{B}}}};
\draw (8.1-7,2.9-1) node[anchor=north west] {\scalebox{1.5}{\textcolor[rgb]{0.00,0.07,1.00}{\textbf{B}}}};
\draw (8.1-7,2.9-1) node[anchor=north west] {\scalebox{1.5}{\textcolor[rgb]{0.00,0.07,1.00}{\textbf{B}}}};
\draw (8.1+4,2.9+1) node[anchor=north west] {\scalebox{1.5}{\textcolor[rgb]{0.00,0.07,1.00}{\textbf{B}}}};
\draw (8.1+2,2.9+1) node[anchor=north west] {\scalebox{1.5}{\textcolor[rgb]{0.00,0.07,1.00}{\textbf{B}}}};
\draw (8.1+4,2.9+1) node[anchor=north west] {\scalebox{1.5}{\textcolor[rgb]{0.00,0.07,1.00}{\textbf{B}}}};
\draw (8.1+3,2.9) node[anchor=north west] {\scalebox{1.5}{\textcolor[rgb]{0.00,0.07,1.00}{\textbf{B}}}};
\draw (8.1-2,2.9-8) node[anchor=north west] {\scalebox{1.5}{\textcolor[rgb]{0.00,0.07,1.00}{\textbf{B}}}};
\draw (8.1-2,2.9-10) node[anchor=north west] {\scalebox{1.5}{\textcolor[rgb]{0.00,0.07,1.00}{\textbf{B}}}};
\draw (8.1-3,2.9-9) node[anchor=north west] {\scalebox{1.5}{\textcolor[rgb]{0.00,0.07,1.00}{\textbf{B}}}};
\draw (8.1-4,2.9+10) node[anchor=north west] {\scalebox{1.5}{\textcolor[rgb]{0.00,0.07,1.00}{\textbf{B}}}};
\draw (8.1-4+2,2.9+10) node[anchor=north west] {\scalebox{1.5}{\textcolor[rgb]{0.00,0.07,1.00}{\textbf{B}}}};
\draw (8.1-5+2,2.9+9) node[anchor=north west] {\scalebox{1.5}{\textcolor[rgb]{0.00,0.07,1.00}{\textbf{B}}}};
\draw (8.1-6+2,2.9+8) node[anchor=north west] {\scalebox{1.5}{\textcolor[rgb]{0.00,0.07,1.00}{\textbf{B}}}};
\draw (8.1-4+2,2.9+8) node[anchor=north west] {\scalebox{1.5}{\textcolor[rgb]{0.00,0.07,1.00}{\textbf{B}}}};
\draw (8.1-5+2,2.9+7) node[anchor=north west] {\scalebox{1.5}{\textcolor[rgb]{0.00,0.07,1.00}{\textbf{B}}}};

\draw (8.1+3,2.9+1) node[anchor=north west] {\scalebox{1.5}{\textcolor[rgb]{1.00,0.00,0.00}{\textbf{R}}}};

\draw (8.1+4,2.9) node[anchor=north west] {\scalebox{1.5}{\textcolor[rgb]{1.00,0.00,0.00}{\textbf{R}}}};
\draw (8.1-8,2.9+1) node[anchor=north west] {\scalebox{1.5}{\textcolor[rgb]{1.00,0.00,0.00}{\textbf{R}}}};
\draw (8.1-7,2.9+2) node[anchor=north west] {\scalebox{1.5}{\textcolor[rgb]{1.00,0.00,0.00}{\textbf{R}}}};
\draw (8.1-7,2.9+2-2) node[anchor=north west] {\scalebox{1.5}{\textcolor[rgb]{1.00,0.00,0.00}{\textbf{R}}}};
\draw (8.1-3,2.9+2-10) node[anchor=north west] {\scalebox{1.5}{\textcolor[rgb]{1.00,0.00,0.00}{\textbf{R}}}};
\draw (8.1-2,2.9+2-11) node[anchor=north west] {\scalebox{1.5}{\textcolor[rgb]{1.00,0.00,0.00}{\textbf{R}}}};
\draw (8.1-3+1,2.9+7) node[anchor=north west] {\scalebox{1.5}{\textcolor[rgb]{1.00,0.00,0.00}{\textbf{R}}}};
\draw (8.1-3,2.9+8) node[anchor=north west] {\scalebox{1.5}{\textcolor[rgb]{1.00,0.00,0.00}{\textbf{R}}}};
\draw (8.1-4,2.9+9) node[anchor=north west] {\scalebox{1.5}{\textcolor[rgb]{1.00,0.00,0.00}{\textbf{R}}}};
\draw (8.1-4+2,2.9+9) node[anchor=north west] {\scalebox{1.5}{\textcolor[rgb]{1.00,0.00,0.00}{\textbf{R}}}};
\draw (8.1-4+1,2.9+10) node[anchor=north west] {\scalebox{1.5}{\textcolor[rgb]{1.00,0.00,0.00}{\textbf{R}}}};

\end{tikzpicture}}
\caption{Discordant options satisfy the F-loss Property with the use of Theorem \ref{thm:patterns}.}
\label{fig:fig7}
\end{center}
\end{figure}

\begin{theorem}\label{thm:numbers}
The game values of {\sc partisan chocolate game} positions are numbers.
\end{theorem}

\begin{proof}
Let $G$ be a {\sc partisan chocolate game} position. Naturally, a pair $(G^L, G^R)$ is either obtained through parallel cuts or perpendicular cuts. If the cuts are concordant (parallel cuts), the pair trivially satisfies the F-loss Property.

To analyze the case where the options are discordant (perpendicular cuts), let us start by proving the simple Claim that states $(2n,0)=(2n+1,1)$. If $n=0$, both $(0,0)$ and $(1,1)$ are $\mathcal{P}$-positions with a game value of zero. For the remaining cases, the game $(2n,0)-(2n+1,1)$ is a $\mathcal{P}$-position because $(2n-(2j+1),0)-(2n+1-(2j+3),1)$ and $(2n-2j,0)-(2n+1-2j,1)$ are $\mathcal{P}$-positions (the first is by Theorem \ref{thm:patterns} iiia and the second is by induction). Therefore, apart from the possible first moves by Left to $(2n,0)-(2n,1)$ or to $(2n,0)-(2n+1,0)$, which have the trivial winning response by Right to $(2n,0)-(2n,0)$, all other initial moves by either player have an available $\mathcal{P}$-position as a response.

If the cuts are discordant, without loss of generality, let us assume that Left cuts vertically and Right cuts horizontally. Let us first consider the case where $G=(n,m)$ with $n+m$ even (the top-right square is red). Consider $G^L=(n-k,m)$ and $G^R=(n,m-k')$, where $k$ is even and $k'$ is odd.

Firstly, suppose that $k'-k\geqslant1$ (that is, $-k-1\geqslant-k'$). Let\linebreak $G^{LR}=(n-k,m-1)$, option that is available, since $m\geqslant1$ to allow Right to make a horizontal cut. We have
$$G^{LR}=(n-k,m-1)\underbrace{=}_{Theorem \ref{thm:patterns},\text{ iiia}}(n,m-k-1)\underbrace{\preccurlyeq}_{Theorem \ref{thm:patterns},\text{ i2a}}(n,m-k')=G^{R}.$$
Secondly, suppose that $k'-k<1$ (that is, $k-k'+1>0$) and $m-k'\geqslant 1$. Let $G^{RL}=(n,m-k'-1)$. We have
$$G^L=(n-k,m)\underbrace{\preccurlyeq}_{Theorem \ref{thm:patterns},\text{ i1a}}(n-k,m+k-k'-1)\underbrace{=}_{Theorem \ref{thm:patterns},\text{ iiib}}(n,m-k'-1)=G^{RL}.$$
Thirdly, suppose that $k'-k<1$ and $m-k'=0$ ($m$ and $n$ are odd). Let $G^{RL}=(n-k+m-2,0)$,  option that is available, since $n$ is odd, and consequently $n-k\geqslant 1$. We have
$$G^L=(n-k,m)\underbrace{=}_{Theorem \ref{thm:patterns},\text{ iiib}}(n-k+m-1,1)\underbrace{=}_{Claim}(n-k+m-2,0)=G^{RL}.$$

Let us consider now the case where $G=(n,m)$ with $n+m$ odd (the top-right square is blue). Consider $G^L=(n-k,m)$ and $G^R=(n,m-k')$, where $k$ is odd and $k'$ is even.

Firstly, suppose that $k-k'>0$. Of course, in this case we have that $n-(k-k')\geqslant n-k\geqslant 0$; according to this, let $G^{RL}=(n-(k-k'),m-k')$. We have
$$G^{L}=(n-k,m)\underbrace{=}_{Theorem \ref{thm:patterns},\text{ iiib}}(n-(k-k'),m-k')=G^{RL}.$$
Secondly, suppose that $k-k'\leqslant0$. Of course, the inequality must be strict because $k$ and $k'$ are different, i.e., $k-k'<0$. Observe that $-1-k\geqslant -k'$ and let $G^{LR}=(n-k,m-1)$. We have
$$G^{LR}=(n-k,m-1)\underbrace{=}_{Theorem \ref{thm:patterns},\text{ iiia}}(n,m-k-1)\underbrace{\preccurlyeq}_{Theorem \ref{thm:patterns},\text{ i2a}}(n,m-k')=G^{RL}.$$

In all cases, the pairs $(G^L, G^R)$  satisfy the F-loss Property. Hence, the game value of $G$ is a number.
\end{proof}

\section{Game values of {\sc partisan chocolate game} positions}
\label{sec:formula}
In order to find a closed formula for the game values of {\sc partisan chocolate game} positions, we begin with the particular case where the chocolate bar has only one column (or one line). This case can be interpreted as an alternating {\sc blue-red hackenbush} string, as stated in the following theorem.

\begin{theorem}\label{thm:isomorphic}
Let $m$ be a nonnegative integer. If $G$ is the $(0,m)$--position (or the $(m,0)$--position) and $H$ is the alternating {\sc blue-red hackenbush} string of size $m$ with a blue edge at the base when there are edges, then $G\cong H$ and $G=H_m=\frac{2^m-(-1)^m}{3\times 2^{m-1}}$.
\end{theorem}

\begin{proof}
If $m=0$, both the $(0,0)$--position and the {\sc blue-red hackenbush} string without edges have the game form $\{\varnothing\! \mid \! \varnothing \}$, and they are trivially isomorphic. This is the base case.

For $m>0$, the Left options of $(0,m)$ are the positions $(0,m')$ with $m'<m$ even, and the Right options of $(0,m)$ are the positions $(0,m'')$\linebreak with $m''<m$ odd. On the other hand, the Left options in the alternating {\sc blue-red hackenbush} string of size $m$ consist of alternating {\sc blue-red hackenbush} strings of size $m'$, where $m'<m$ is even, and the Right options consist of alternating {\sc blue-red hackenbush} strings of size $m''$, where $m''<m$ is odd. By induction, these options are pairwise isomorphic, making $G$ and $H$, in turn, also isomorphic.

The game value $G=H_m=\frac{2^m-(-1)^m}{3\times 2^{m-1}}$ is a straightforward consequence of the fact that the game values of alternating {\sc blue-red hackenbush} strings follow this formula.
\end{proof}

As we will see, the main result is closely related to the sequence $H_n$. To prepare for its proof, some preliminary lemmas are required, which we prove below.

\begin{lemma}\label{lem:facts}
If $n$ is a nonnegative integer, we have the following facts:
     \begin{itemize}
        \item $H_{2n}=\frac{4^n-1}{6\times 4^{n-1}}$
        \item $H_{2n+1}=\frac{2\times4^n+1}{3\times 4^{n}}$
        \item $H_{2n}<H_{2(n+1)}$
        \item $H_{2n+1}>H_{2(n+1)+1}$
    \end{itemize}
\end{lemma}

\begin{proof}
The first two items result from straightforward replacements by $2n$ and $2n+1$ in $\frac{2^n-(-1)^n}{3\times 2^{n-1}}$, followed by subsequent algebraic simplifications.

The last two items follow from the first two items and the facts\linebreak $H_{2(n+1)}-H_{2n}=\frac{1}{2\times4^n}>0$ and $H_{2(n+1)+1}-H_{2n+1}=-\frac{1}{4^{n+1}}<0$.
\end{proof}

\begin{lemma}\label{lem:facts2}
If $n$ is a nonnegative integer, we have the following facts:
     \begin{itemize}
        \item $\{H_{2n}\,|\,H_{2n+1}\}=H_{2n+2}$
        \item $\{H_{2n+2}\,|\,H_{2n+1}\}=H_{2n+3}$
        \item $\{H_{2n}\,|\,H_{2n+3}\}=H_{2n+2}$
    \end{itemize}
\end{lemma}

\begin{proof}
All items will be proved by using the Simplicity Rule (Theorem \ref{thm:simplicity}).
     \begin{itemize}
        \item From Lemma \ref{lem:facts}, we have that
$$\{H_{2n}\,|\,H_{2n+1}\}=\left\{\frac{4^n-1}{6\times 4^{n-1}}\,\middle|\,\frac{2\times4^n+1}{3\times 4^{n}}\right\}=\left\{\frac{4^{n+1}-4}{6\times 4^{n}}\,\middle|\,\frac{4^{n+1}+2}{6\times 4^{n}}\right\}.$$
From the integers strictly between $4^{n+1}-4$ and $4^{n+1}+2$, we need to find a multiple of $3$ to obtain a dyadic strictly between the Left option and the Right option. The only multiple of $3$ under these conditions is $4^{n+1}-1$. Thus, the desired simplest number is $H_{2n+2}=\frac{4^{n+1}-1}{6\times 4^{n}}$.
        \item From Lemma \ref{lem:facts}, we have that
\begin{align*}
\{H_{2n+2}\,|\,H_{2n+1}\} &= \left\{\frac{4^{n+1}-1}{6\times 4^{n}}\,\middle|\,\frac{2\times4^{n}+1}{3\times 4^{n}}\right\} \\
&= \left\{\frac{2\times4^{n+1}-2}{12\times 4^{n}}\,\middle|\,\frac{2\times4^{n+1}+4}{12\times 4^{n}}\right\}.
\end{align*}
From the integers strictly between $2\times4^{n+1}-2$ and $2\times4^{n+1}+4$, we need to find a multiple of $3$ to obtain a dyadic strictly between the Left option and the Right option. The only multiple of $3$ under these conditions is $2\times 4^{n+1}+1$. Thus, the desired simplest number is $H_{2n+3}=\frac{2\times 4^{n+1}+1}{12\times 4^{n}}$.
        \item From Lemma \ref{lem:facts}, we have that
\begin{align*}
\{H_{2n}\,|\,H_{2n+3}\} &= \left\{\frac{4^n-1}{6\times 4^{n-1}}\,\middle|\,\frac{2\times4^{n+1}+1}{3\times 4^{n+1}}\right\} \\
&= \left\{\frac{2\times 4^{n+1}-8}{12\times 4^{n}}\,\middle|\,\frac{2\times 4^{n+1}+1}{12\times 4^{n}}\right\}.
\end{align*}
From the integers strictly between $2\times 4^{n+1}-8$ and $2\times 4^{n+1}+1$, we need to find a multiple of $3$ to obtain a dyadic strictly between the Left option and the Right option. In this case, there is more than one multiple of $3$, but of these, only $2\times 4^{n+1}-2$ is a multiple of $6$. Thus, the desired simplest number is $H_{2n+2}=\frac{2\times 4^{n+1}-2}{12\times 4^{n}}$.
    \end{itemize}
\end{proof}

The following theorem constitutes the main result of this document.

\begin{theorem}[Game Values for $(n,m)$]\label{thm:main}
For nonnegative integers $n$ and $m$, the following results hold:
\begin{enumerate}
\item If either $n$ or $m$ (or both) is even, then $(n,m) = H_{n+m}$.
\item If both $n$ and $m$ are odd, then $(n,m) = H_{n+m-2}$.
\end{enumerate}
\end{theorem}

\begin{observation}
In this theorem, both $(n, m)$ and $H_{n+m}$, as well as $H_{n+m-2}$, represent game positions, and ``$=$'' denotes the game equivalence.
\end{observation}

\begin{proof}
If either $n=0$ or $m=0$, the result follows as a consequence of Theorem \ref{thm:isomorphic}.

Suppose now that both $n$ and $m$ are even integers. By Theorem \ref{thm:patterns} and Domination, we only need to consider the Left options $(n-2,m)$ and\linebreak $(n,m-2).$ By induction, we can establish that both have identical game values, which are equal to $H_{n+m-2}$. On the other hand, also by Theorem \ref{thm:patterns} and Domination, we only need to consider the Right options $(n-1,m)$ and $(n,m-1).$ By induction, we can establish that both have identical game values, which are equal to $H_{n+m-1}$. Therefore, by Lemma \ref{lem:facts2} (first item), $$(n,m)=\{H_{n+m-2}\,|\,H_{n+m-1}\}=H_{n+m}.$$

Suppose that $n$ and $m$ have different parities. By Theorem \ref{thm:patterns} and Domination, we only need to consider the Left options $(n-1,m)$ and $(n,m-1).$ By induction, we can establish that one of the values of these options is $H_{n+m-1}$ and the other is $H_{n+m-3}$, depending on the analysis of the parities. As a result of Domination, $H_{n+m-3}$ is removed. On the other hand, also by Theorem \ref{thm:patterns} and Domination, we only need to consider the Right options $(n-2,m)$ and $(n,m-2).$ By induction, we can establish that both have identical game values, which are equal to $H_{n+m-2}$. Therefore, by Lemma \ref{lem:facts2} (second item), $$(n,m)=\{H_{n+m-1}\,|\,H_{n+m-2}\}=H_{n+m}.$$

Finally, suppose that both $n$ and $m$ are odd integers. By Theorem \ref{thm:patterns} and Domination, we only need to consider the Left options $(n-2,m)$ and $(n,m-2).$ By induction, we can establish that both have identical game values, which are equal to $H_{n+m-4}$. On the other hand, also by Theorem~\ref{thm:patterns} and Domination, we only need to consider the Right options $(n-1,m)$ and $(n,m-1).$ By induction, we can establish that both have identical game values, which are equal to $H_{n+m-1}$. Therefore, by Lemma \ref{lem:facts2} (third item), $$(n,m)=\{H_{n+m-4}\,|\,H_{n+m-1}\}=H_{n+m-2}.$$
It is important noting that there is a configuration where the Left player has no options, which corresponds to the position $(1,1)$. In this case, since it is a $\mathcal{P}$-position, the game value is zero, consistent with the statement.
\end{proof}

Table \ref{tab:table} illustrates the pattern regarding the game values of the {\sc partisan chocolate game} positions.

\begin{table}[!htbp]
\begin{center}
\begin{tabular}{|c|c|c|c|c|c|c|c|c|c|}
 \hline
\rule{0pt}{3ex}\bm{\textcolor[rgb]{0.00,0.07,1.00}{\raisebox{0.5ex}{$\frac{171}{256}$}}}&
\rule{0pt}{3ex}\bm{\textcolor[rgb]{1.00,0.00,0.00}{\raisebox{0.5ex}{$\frac{85}{128}$}}} &
\rule{0pt}{3ex}\bm{\textcolor[rgb]{0.00,0.07,1.00}{\raisebox{0.5ex}{$\frac{683}{1024}$}}} &
\rule{0pt}{3ex}\bm{\textcolor[rgb]{1.00,0.00,0.00}{\raisebox{0.5ex}{$\frac{341}{512}$}}} &
\rule{0pt}{3ex}\bm{\textcolor[rgb]{0.00,0.07,1.00}{\raisebox{0.5ex}{$\frac{2731}{4096}$}}} &
\rule{0pt}{3ex}\bm{\textcolor[rgb]{1.00,0.00,0.00}{\raisebox{0.5ex}{$\frac{1365}{2048}$}}}&
\rule{0pt}{3ex}\bm{\textcolor[rgb]{0.00,0.07,1.00}{\raisebox{0.5ex}{$\frac{10923}{16384}$}}} &
\rule{0pt}{3ex}\bm{\textcolor[rgb]{1.00,0.00,0.00}{\raisebox{0.5ex}{$\frac{5461}{8192}$}}}&
\rule{0pt}{3ex}\bm{\textcolor[rgb]{0.00,0.07,1.00}{\raisebox{0.5ex}{$\frac{43691}{65536}$}}} &
\rule{0pt}{3ex}\bm{\textcolor[rgb]{1.00,0.00,0.00}{\raisebox{0.5ex}{$\frac{21845}{32768}$}}} \\
  \hline
\rule{0pt}{3ex}\bm{\textcolor[rgb]{1.00,0.00,0.00}{\raisebox{0.5ex}{$\frac{85}{128}$}}}&
\rule{0pt}{3ex}\bm{\textcolor[rgb]{0.00,0.07,1.00}{\raisebox{0.5ex}{$\frac{171}{256}$}}} &
\rule{0pt}{3ex}\bm{\textcolor[rgb]{1.00,0.00,0.00}{\raisebox{0.5ex}{$\frac{341}{512}$}}} &
\rule{0pt}{3ex}\bm{\textcolor[rgb]{0.00,0.07,1.00}{\raisebox{0.5ex}{$\frac{683}{1024}$}}} &
\rule{0pt}{3ex}\bm{\textcolor[rgb]{1.00,0.00,0.00}{\raisebox{0.5ex}{$\frac{1365}{2048}$}}} &
\rule{0pt}{3ex}\bm{\textcolor[rgb]{0.00,0.07,1.00}{\raisebox{0.5ex}{$\frac{2731}{4096}$}}} &
\rule{0pt}{3ex}\bm{\textcolor[rgb]{1.00,0.00,0.00}{\raisebox{0.5ex}{$\frac{5461}{8192}$}}} &
\rule{0pt}{3ex}\bm{\textcolor[rgb]{0.00,0.07,1.00}{\raisebox{0.5ex}{$\frac{10923}{16384}$}}}  &
\rule{0pt}{3ex}\bm{\textcolor[rgb]{1.00,0.00,0.00}{\raisebox{0.5ex}{$\frac{21845}{32768}$}}} &
\rule{0pt}{3ex}\bm{\textcolor[rgb]{0.00,0.07,1.00}{\raisebox{0.5ex}{$\frac{43691}{65536}$}}} \\
 \hline
\rule{0pt}{3ex}\bm{\textcolor[rgb]{0.00,0.07,1.00}{\raisebox{0.5ex}{$\frac{43}{64}$}}}&
\rule{0pt}{3ex}\bm{\textcolor[rgb]{1.00,0.00,0.00}{\raisebox{0.5ex}{$\frac{21}{32}$}}} &
\rule{0pt}{3ex}\bm{\textcolor[rgb]{0.00,0.07,1.00}{\raisebox{0.5ex}{$\frac{171}{256}$}}} &
\rule{0pt}{3ex}\bm{\textcolor[rgb]{1.00,0.00,0.00}{\raisebox{0.5ex}{$\frac{85}{128}$}}} &
\rule{0pt}{3ex}\bm{\textcolor[rgb]{0.00,0.07,1.00}{\raisebox{0.5ex}{$\frac{683}{1024}$}}} &
\rule{0pt}{3ex}\bm{\textcolor[rgb]{1.00,0.00,0.00}{\raisebox{0.5ex}{$\frac{341}{512}$}}} &
\rule{0pt}{3ex}\bm{\textcolor[rgb]{0.00,0.07,1.00}{\raisebox{0.5ex}{$\frac{2731}{4096}$}}} &
\rule{0pt}{3ex}\bm{\textcolor[rgb]{1.00,0.00,0.00}{\raisebox{0.5ex}{$\frac{1365}{2048}$}}} &
\rule{0pt}{3ex}\bm{\textcolor[rgb]{0.00,0.07,1.00}{\raisebox{0.5ex}{$\frac{10923}{16384}$}}} &
\rule{0pt}{3ex}\bm{\textcolor[rgb]{1.00,0.00,0.00}{\raisebox{0.5ex}{$\frac{5461}{8192}$}}} \\
  \hline
\rule{0pt}{3ex}\bm{\textcolor[rgb]{1.00,0.00,0.00}{\raisebox{0.5ex}{$\frac{21}{32}$}}} &
\rule{0pt}{3ex}\bm{\textcolor[rgb]{0.00,0.07,1.00}{\raisebox{0.5ex}{$\frac{43}{64}$}}} &
\rule{0pt}{3ex}\bm{\textcolor[rgb]{1.00,0.00,0.00}{\raisebox{0.5ex}{$\frac{85}{128}$}}} &
\rule{0pt}{3ex}\bm{\textcolor[rgb]{0.00,0.07,1.00}{\raisebox{0.5ex}{$\frac{171}{256}$}}} &
\rule{0pt}{3ex}\bm{\textcolor[rgb]{1.00,0.00,0.00}{\raisebox{0.5ex}{$\frac{341}{512}$}}} &
\rule{0pt}{3ex}\bm{\textcolor[rgb]{0.00,0.07,1.00}{\raisebox{0.5ex}{$\frac{683}{1024}$}}} &
\rule{0pt}{3ex}\bm{\textcolor[rgb]{1.00,0.00,0.00}{\raisebox{0.5ex}{$\frac{1365}{2048}$}}} &
\rule{0pt}{3ex}\bm{\textcolor[rgb]{0.00,0.07,1.00}{\raisebox{0.5ex}{$\frac{2731}{4096}$}}}  &
\rule{0pt}{3ex}\bm{\textcolor[rgb]{1.00,0.00,0.00}{\raisebox{0.5ex}{$\frac{5461}{8192}$}}} &
\rule{0pt}{3ex}\bm{\textcolor[rgb]{0.00,0.07,1.00}{\raisebox{0.5ex}{$\frac{10923}{16384}$}}} \\
 \hline
\rule{0pt}{3ex}\bm{\textcolor[rgb]{0.00,0.07,1.00}{\raisebox{0.5ex}{$\frac{11}{16}$}}} &
\rule{0pt}{3ex}\bm{\textcolor[rgb]{1.00,0.00,0.00}{\raisebox{0.5ex}{$\frac{5}{8}$}}} &
\rule{0pt}{3ex}\bm{\textcolor[rgb]{0.00,0.07,1.00}{\raisebox{0.5ex}{$\frac{43}{64}$}}} &
\rule{0pt}{3ex}\bm{\textcolor[rgb]{1.00,0.00,0.00}{\raisebox{0.5ex}{$\frac{21}{32}$}}} &
\rule{0pt}{3ex}\bm{\textcolor[rgb]{0.00,0.07,1.00}{\raisebox{0.5ex}{$\frac{171}{256}$}}} &
\rule{0pt}{3ex}\bm{\textcolor[rgb]{1.00,0.00,0.00}{\raisebox{0.5ex}{$\frac{85}{128}$}}} &
\rule{0pt}{3ex}\bm{\textcolor[rgb]{0.00,0.07,1.00}{\raisebox{0.5ex}{$\frac{683}{1024}$}}} &
\rule{0pt}{3ex}\bm{\textcolor[rgb]{1.00,0.00,0.00}{\raisebox{0.5ex}{$\frac{341}{512}$}}} &
\rule{0pt}{3ex}\bm{\textcolor[rgb]{0.00,0.07,1.00}{\raisebox{0.5ex}{$\frac{2731}{4096}$}}} &
\rule{0pt}{3ex}\bm{\textcolor[rgb]{1.00,0.00,0.00}{\raisebox{0.5ex}{$\frac{1365}{2048}$}}} \\
  \hline
\rule{0pt}{3ex}\bm{\textcolor[rgb]{1.00,0.00,0.00}{\raisebox{0.5ex}{$\frac{5}{8}$}}} &
\rule{0pt}{3ex}\bm{\textcolor[rgb]{0.00,0.07,1.00}{\raisebox{0.5ex}{$\frac{11}{16}$}}} &
\rule{0pt}{3ex}\bm{\textcolor[rgb]{1.00,0.00,0.00}{\raisebox{0.5ex}{$\frac{21}{32}$}}} &
\rule{0pt}{3ex}\bm{\textcolor[rgb]{0.00,0.07,1.00}{\raisebox{0.5ex}{$\frac{43}{64}$}}} &
\rule{0pt}{3ex}\bm{\textcolor[rgb]{1.00,0.00,0.00}{\raisebox{0.5ex}{$\frac{85}{128}$}}} &
\rule{0pt}{3ex}\bm{\textcolor[rgb]{0.00,0.07,1.00}{\raisebox{0.5ex}{$\frac{171}{256}$}}} &
\rule{0pt}{3ex}\bm{\textcolor[rgb]{1.00,0.00,0.00}{\raisebox{0.5ex}{$\frac{341}{512}$}}} &
\rule{0pt}{3ex}\bm{\textcolor[rgb]{0.00,0.07,1.00}{\raisebox{0.5ex}{$\frac{683}{1024}$}}} &
\rule{0pt}{3ex}\bm{\textcolor[rgb]{1.00,0.00,0.00}{\raisebox{0.5ex}{$\frac{1365}{2048}$}}} &
\rule{0pt}{3ex}\bm{\textcolor[rgb]{0.00,0.07,1.00}{\raisebox{0.5ex}{$\frac{2731}{4096}$}}} \\
  \hline
\rule{0pt}{3ex}\bm{\textcolor[rgb]{0.00,0.07,1.00}{\raisebox{0.5ex}{$\frac{3}{4}$}}} &
\rule{0pt}{3ex}\bm{\textcolor[rgb]{1.00,0.00,0.00}{\raisebox{0.5ex}{$\frac{1}{2}$}}} &
\rule{0pt}{3ex}\bm{\textcolor[rgb]{0.00,0.07,1.00}{\raisebox{0.5ex}{$\frac{11}{16}$}}} &
\rule{0pt}{3ex}\bm{\textcolor[rgb]{1.00,0.00,0.00}{\raisebox{0.5ex}{$\frac{5}{8}$}}} &
\rule{0pt}{3ex}\bm{\textcolor[rgb]{0.00,0.07,1.00}{\raisebox{0.5ex}{$\frac{43}{64}$}}} &
\rule{0pt}{3ex}\bm{\textcolor[rgb]{1.00,0.00,0.00}{\raisebox{0.5ex}{$\frac{21}{32}$}}} &
\rule{0pt}{3ex}\bm{\textcolor[rgb]{0.00,0.07,1.00}{\raisebox{0.5ex}{$\frac{171}{256}$}}} &
\rule{0pt}{3ex}\bm{\textcolor[rgb]{1.00,0.00,0.00}{\raisebox{0.5ex}{$\frac{85}{128}$}}} &
\rule{0pt}{3ex}\bm{\textcolor[rgb]{0.00,0.07,1.00}{\raisebox{0.5ex}{$\frac{683}{1024}$}}} &
\rule{0pt}{3ex}\bm{\textcolor[rgb]{1.00,0.00,0.00}{\raisebox{0.5ex}{$\frac{341}{512}$}}} \\
  \hline
\rule{0pt}{3ex}\bm{\textcolor[rgb]{1.00,0.00,0.00}{\raisebox{0.5ex}{$\frac{1}{2}$}}} &
\rule{0pt}{3ex}\bm{\textcolor[rgb]{0.00,0.07,1.00}{\raisebox{0.5ex}{$\frac{3}{4}$}}} &
\rule{0pt}{3ex}\bm{\textcolor[rgb]{1.00,0.00,0.00}{\raisebox{0.5ex}{$\frac{5}{8}$}}} &
\rule{0pt}{3ex}\bm{\textcolor[rgb]{0.00,0.07,1.00}{\raisebox{0.5ex}{$\frac{11}{16}$}}} &
\rule{0pt}{3ex}\bm{\textcolor[rgb]{1.00,0.00,0.00}{\raisebox{0.5ex}{$\frac{21}{32}$}}} &
\rule{0pt}{3ex}\bm{\textcolor[rgb]{0.00,0.07,1.00}{\raisebox{0.5ex}{$\frac{43}{64}$}}} &
\rule{0pt}{3ex}\bm{\textcolor[rgb]{1.00,0.00,0.00}{\raisebox{0.5ex}{$\frac{85}{128}$}}} &
\rule{0pt}{3ex}\bm{\textcolor[rgb]{0.00,0.07,1.00}{\raisebox{0.5ex}{$\frac{171}{256}$}}} &
\rule{0pt}{3ex}\bm{\textcolor[rgb]{1.00,0.00,0.00}{\raisebox{0.5ex}{$\frac{341}{512}$}}} &
\rule{0pt}{3ex}\bm{\textcolor[rgb]{0.00,0.07,1.00}{\raisebox{0.5ex}{$\frac{683}{1024}$}}} \\
  \hline
\rule{0pt}{3ex}\bm{\textcolor[rgb]{0.00,0.07,1.00}{$1$}} &
\rule{0pt}{3ex}\bm{\textcolor[rgb]{1.00,0.00,0.00}{$0$}} &
\rule{0pt}{3ex}\bm{\textcolor[rgb]{0.00,0.07,1.00}{\raisebox{0.5ex}{$\frac{3}{4}$}}} &
\rule{0pt}{3ex}\bm{\textcolor[rgb]{1.00,0.00,0.00}{\raisebox{0.5ex}{$\frac{1}{2}$}}} &
\rule{0pt}{3ex}\bm{\textcolor[rgb]{0.00,0.07,1.00}{\raisebox{0.5ex}{$\frac{11}{16}$}}} &
\rule{0pt}{3ex}\bm{\textcolor[rgb]{1.00,0.00,0.00}{\raisebox{0.5ex}{$\frac{5}{8}$}}} &
\rule{0pt}{3ex}\bm{\textcolor[rgb]{0.00,0.07,1.00}{\raisebox{0.5ex}{$\frac{43}{64}$}}} &
\rule{0pt}{3ex}\bm{\textcolor[rgb]{1.00,0.00,0.00}{\raisebox{0.5ex}{$\frac{21}{32}$}}} &
\rule{0pt}{3ex}\bm{\textcolor[rgb]{0.00,0.07,1.00}{\raisebox{0.5ex}{$\frac{171}{256}$}}} &
\rule{0pt}{3ex}\bm{\textcolor[rgb]{1.00,0.00,0.00}{\raisebox{0.5ex}{$\frac{85}{128}$}}} \\
\hline
\rule{0pt}{3ex}\bm{$0$} &
\rule{0pt}{3ex}\bm{\textcolor[rgb]{0.00,0.07,1.00}{$1$}} &
\rule{0pt}{3ex}\bm{\textcolor[rgb]{1.00,0.00,0.00}{\raisebox{0.5ex}{$\frac{1}{2}$}}} &
\rule{0pt}{3ex}\bm{\textcolor[rgb]{0.00,0.07,1.00}{\raisebox{0.5ex}{$\frac{3}{4}$}}} &
\rule{0pt}{3ex}\bm{\textcolor[rgb]{1.00,0.00,0.00}{\raisebox{0.5ex}{$\frac{5}{8}$}}} &
\rule{0pt}{3ex}\bm{\textcolor[rgb]{0.00,0.07,1.00}{\raisebox{0.5ex}{$\frac{11}{16}$}}} &
\rule{0pt}{3ex}\bm{\textcolor[rgb]{1.00,0.00,0.00}{\raisebox{0.5ex}{$\frac{21}{32}$}}} &
\rule{0pt}{3ex}\bm{\textcolor[rgb]{0.00,0.07,1.00}{\raisebox{0.5ex}{$\frac{43}{64}$}}} &
\rule{0pt}{3ex}\bm{\textcolor[rgb]{1.00,0.00,0.00}{\raisebox{0.5ex}{$\frac{85}{128}$}}} &
\rule{0pt}{3ex}\bm{\textcolor[rgb]{0.00,0.07,1.00}{\raisebox{0.5ex}{$\frac{171}{256}$}}} \\
\hline
\end{tabular}
\caption{Game values of the {\sc partisan chocolate game} positions.}
\label{tab:table}
\end{center}
\end{table}

\section{Game practice}
\label{sec:practice}

This section serves to briefly show how Theorem \ref{thm:main} can be applied in game practice. As an example, Figure \ref{fig:fig8} exhibits the disjunctive sum\linebreak $-(2,4)-(1,3)+(2,3)+(2,0)$. Since the squares adjacent to the poisoned square are red, the first two components have negative game values. By Theorem \ref{thm:main}, the sum is equal to $-\frac{21}{32}-\frac{1}{2}+\frac{11}{16}+\frac{1}{2}=\frac{1}{32}>0$ and, consequently, Left has a winning move.

\begin{figure}[htb!]
\begin{center}
\scalebox{0.4}{
\begin{tikzpicture}
\clip(0.8,0.8) rectangle (15.2,6.2);
\fill[line width=0.pt,fill=black,fill opacity=1.0] (1.,1.) -- (2.,1.) -- (2.,2.) -- (1.,2.) -- cycle;
\fill[line width=0.pt,fill=black,fill opacity=1.0] (5.,1.) -- (6.,1.) -- (6.,2.) -- (5.,2.) -- cycle;
\fill[line width=0.pt,fill=black,fill opacity=1.0] (12.,1.) -- (13.,1.) -- (13.,2.) -- (12.,2.) -- cycle;
\fill[line width=0.pt,fill=black,fill opacity=1.0] (9.,2.) -- (8.,2.) -- (8.,1.) -- (9.,1.) -- cycle;
\draw [line width=2.pt] (1.,1.)-- (1.,6.);
\draw [line width=2.pt] (1.,1.)-- (4.,1.);
\draw [line width=2.pt] (4.,1.)-- (4.,6.);
\draw [line width=2.pt] (1.,6.)-- (4.,6.);
\draw [line width=2.pt] (5.,1.)-- (5.,5.);
\draw [line width=2.pt] (5.,5.)-- (7.,5.);
\draw [line width=2.pt] (5.,1.)-- (7.,1.);
\draw [line width=2.pt] (7.,1.)-- (7.,5.);
\draw [line width=2.pt] (8.,1.)-- (8.,5.);
\draw [line width=2.pt] (12.,1.)-- (12.,2.);
\draw [line width=2.pt] (12.,2.)-- (15.,2.);
\draw [line width=2.pt] (15.,2.)-- (15.,1.);
\draw [line width=2.pt] (15.,1.)-- (12.,1.);
\draw [line width=2.pt] (1.,2.)-- (4.,2.);
\draw [line width=2.pt] (1.,3.)-- (4.,3.);
\draw [line width=2.pt] (1.,4.)-- (4.,4.);
\draw [line width=2.pt] (1.,5.)-- (4.,5.);
\draw [line width=2.pt] (2.,6.)-- (2.,1.);
\draw [line width=2.pt] (3.,6.)-- (3.,1.);
\draw [line width=2.pt] (6.,5.)-- (6.,1.);
\draw [line width=2.pt] (5.,4.)-- (7.,4.);
\draw [line width=2.pt] (5.,3.)-- (7.,3.);
\draw [line width=2.pt] (5.,2.)-- (7.,2.);
\draw [line width=2.pt] (13.,1.)-- (13.,2.);
\draw [line width=2.pt] (14.,1.)-- (14.,2.);
\draw [line width=2.pt] (9.,5.)-- (9.,1.);
\draw [line width=2.pt] (10.,1.)-- (10.,5.);
\draw [line width=2.pt] (11.,1.)-- (11.,5.);
\draw [line width=2.pt] (11.,5.)-- (8.,5.);
\draw [line width=2.pt] (8.,4.)-- (11.,4.);
\draw [line width=2.pt] (8.,3.)-- (11.,3.);
\draw [line width=2.pt] (8.,2.)-- (11.,2.);
\draw [line width=2.pt] (11.,1.)-- (8.,1.);

\draw (8.1-7,2.9+2) node[anchor=north west] {\scalebox{1.5}{\textcolor[rgb]{1.00,0.00,0.00}{\textbf{R}}}};
\draw (8.1-7,2.9+2-2) node[anchor=north west] {\scalebox{1.5}{\textcolor[rgb]{1.00,0.00,0.00}{\textbf{R}}}};
\draw (8.1-6,2.9+2-3) node[anchor=north west] {\scalebox{1.5}{\textcolor[rgb]{1.00,0.00,0.00}{\textbf{R}}}};
\draw (8.1-6,2.9+2-1) node[anchor=north west] {\scalebox{1.5}{\textcolor[rgb]{1.00,0.00,0.00}{\textbf{R}}}};
\draw (8.1-6,2.9+2+1) node[anchor=north west] {\scalebox{1.5}{\textcolor[rgb]{1.00,0.00,0.00}{\textbf{R}}}};
\draw (8.1-5,2.9+2) node[anchor=north west] {\scalebox{1.5}{\textcolor[rgb]{1.00,0.00,0.00}{\textbf{R}}}};
\draw (8.1-5,2.9) node[anchor=north west] {\scalebox{1.5}{\textcolor[rgb]{1.00,0.00,0.00}{\textbf{R}}}};
\draw (8.1-5,2.9-2) node[anchor=north west] {\scalebox{1.5}{\textcolor[rgb]{1.00,0.00,0.00}{\textbf{R}}}};
\draw (8.1-3,2.9) node[anchor=north west] {\scalebox{1.5}{\textcolor[rgb]{1.00,0.00,0.00}{\textbf{R}}}};
\draw (8.1-3,2.9+2) node[anchor=north west] {\scalebox{1.5}{\textcolor[rgb]{1.00,0.00,0.00}{\textbf{R}}}};
\draw (8.1-2,2.9+1) node[anchor=north west] {\scalebox{1.5}{\textcolor[rgb]{1.00,0.00,0.00}{\textbf{R}}}};
\draw (8.1-2,2.9-1) node[anchor=north west] {\scalebox{1.5}{\textcolor[rgb]{1.00,0.00,0.00}{\textbf{R}}}};
\draw (8.1+2,2.9-1) node[anchor=north west] {\scalebox{1.5}{\textcolor[rgb]{1.00,0.00,0.00}{\textbf{R}}}};
\draw (8.1+2,2.9+1) node[anchor=north west] {\scalebox{1.5}{\textcolor[rgb]{1.00,0.00,0.00}{\textbf{R}}}};
\draw (8.1+1,2.9+2) node[anchor=north west] {\scalebox{1.5}{\textcolor[rgb]{1.00,0.00,0.00}{\textbf{R}}}};
\draw (8.1+1,2.9) node[anchor=north west] {\scalebox{1.5}{\textcolor[rgb]{1.00,0.00,0.00}{\textbf{R}}}};
\draw (8.1,2.9+1) node[anchor=north west] {\scalebox{1.5}{\textcolor[rgb]{1.00,0.00,0.00}{\textbf{R}}}};
\draw (8.1+6,2.9-1) node[anchor=north west] {\scalebox{1.5}{\textcolor[rgb]{1.00,0.00,0.00}{\textbf{R}}}};

\draw (8.1-7,2.9+3) node[anchor=north west] {\scalebox{1.5}{\textcolor[rgb]{0.00,0.07,1.00}{\textbf{B}}}};
\draw (8.1-7,2.9+1) node[anchor=north west] {\scalebox{1.5}{\textcolor[rgb]{0.00,0.07,1.00}{\textbf{B}}}};
\draw (8.1-6,2.9) node[anchor=north west] {\scalebox{1.5}{\textcolor[rgb]{0.00,0.07,1.00}{\textbf{B}}}};
\draw (8.1-6,2.9+2) node[anchor=north west] {\scalebox{1.5}{\textcolor[rgb]{0.00,0.07,1.00}{\textbf{B}}}};
\draw (8.1-5,2.9+3) node[anchor=north west] {\scalebox{1.5}{\textcolor[rgb]{0.00,0.07,1.00}{\textbf{B}}}};
\draw (8.1-5,2.9+1) node[anchor=north west] {\scalebox{1.5}{\textcolor[rgb]{0.00,0.07,1.00}{\textbf{B}}}};
\draw (8.1-5,2.9-1) node[anchor=north west] {\scalebox{1.5}{\textcolor[rgb]{0.00,0.07,1.00}{\textbf{B}}}};
\draw (8.1-2,2.9) node[anchor=north west] {\scalebox{1.5}{\textcolor[rgb]{0.00,0.07,1.00}{\textbf{B}}}};
\draw (8.1-2,2.9+2) node[anchor=north west] {\scalebox{1.5}{\textcolor[rgb]{0.00,0.07,1.00}{\textbf{B}}}};
\draw (8.1-3,2.9+1) node[anchor=north west] {\scalebox{1.5}{\textcolor[rgb]{0.00,0.07,1.00}{\textbf{B}}}};
\draw (8.1+1,2.9+1) node[anchor=north west] {\scalebox{1.5}{\textcolor[rgb]{0.00,0.07,1.00}{\textbf{B}}}};
\draw (8.1+1,2.9-1) node[anchor=north west] {\scalebox{1.5}{\textcolor[rgb]{0.00,0.07,1.00}{\textbf{B}}}};
\draw (8.1+2,2.9) node[anchor=north west] {\scalebox{1.5}{\textcolor[rgb]{0.00,0.07,1.00}{\textbf{B}}}};
\draw (8.1+2,2.9+2) node[anchor=north west] {\scalebox{1.5}{\textcolor[rgb]{0.00,0.07,1.00}{\textbf{B}}}};
\draw (8.1,2.9+2) node[anchor=north west] {\scalebox{1.5}{\textcolor[rgb]{0.00,0.07,1.00}{\textbf{B}}}};
\draw (8.1,2.9) node[anchor=north west] {\scalebox{1.5}{\textcolor[rgb]{0.00,0.07,1.00}{\textbf{B}}}};
\draw (8.1+5,2.9-1) node[anchor=north west] {\scalebox{1.5}{\textcolor[rgb]{0.00,0.07,1.00}{\textbf{B}}}};
\end{tikzpicture}}
\caption{How can Left win the disjunctive sum $-(2,4)-(1,3)+(2,3)+(2,0)$?}
\label{fig:fig8}
\end{center}
\end{figure}
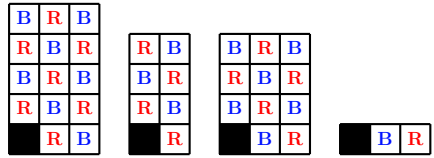

It is well-known in CGT that in these endgames, a player should choose a component with the highest possible power of $2$ in the denominator. In this case, in order to win, Left should make her move in the first component. She can do this by removing the rightmost column or the topmost row. Since Left is greedy, she will choose the column because it has more chocolate blocks.

{\sc partisan chocolate game}  has an interesting implementation on a traditional chessboard. The game is played with white and black rooks that move horizontally and vertically, just like in {\sc chess}. However, the rooks can only move to the left or down. If a rook is black, Left can only move the rook to black squares, while Right can only move it to white squares. If a rook is white, Left can only move it to white squares, and Right can only move it to black squares. The rooks do not conflict in any way; they can occupy the same squares and can jump over each other without blocking.

Figure \ref{fig:fig9} shows the disjunctive sum $-(2,4)-(1,3)+(2,3)+(2,0)$ once again, but this time with the chessboard implementation. Can the reader find a winning move for Left?

\begin{figure}[!htb]
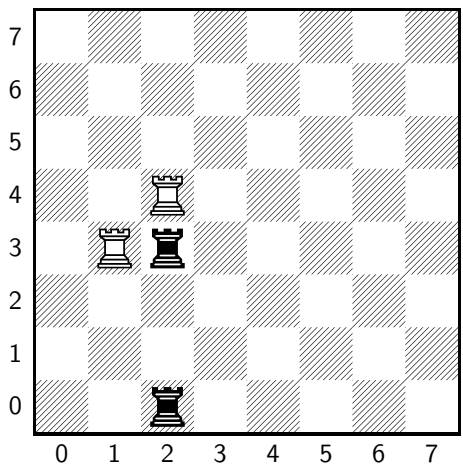

\begin{center}
\setchessboard{setpieces=\mylist}
\def\mylist{Rb4, Ra3, rb0, rb3}
\chessboard[zero,labelbottomformat=\arabic{filelabel},maxfield=g7,showmover=false]\\
\end{center}
\caption{The disjunctive sum $-(2,4)-(1,3)+(2,3)+(2,0)$ on a chessboard.}
\label{fig:fig9}
\end{figure}

\section{Final remarks}
\label{sec:final}

There are at least two natural future explorations to consider. One of them involves the analysis of the three-dimensional version of {\sc partisan chocolate game}, which seems to have a significant level of difficulty. Another one involves adding green chocolate blocks (mint flavor). In CGT, traditionally, green is the impartial color. Therefore, in this version, both players can cut the chocolate along a vertical line, provided that the top square of the column immediately to the right of that line is green, or they can cut the chocolate along a horizontal line, as long as the rightmost square in the row just above that line is green. Inevitably, the ruleset loses the property of all game values being numbers and also presents non-trivial open problems.

Part of the work presented here was developed at the \emph{Combinatorial Game Theory Colloquium IV} (Azores, Portugal, January 23--25, 2023), a renowned conference dedicated to CGT.

Additionally, the authors would like to express their gratitude to Dr. Ryohei Miyadera and Dr. Ko Sakai for their valuable suggestions.

\end{document}